\documentclass[final]{siamltex}

\usepackage{graphicx}
\usepackage{amssymb,amsmath}
\usepackage{comment}
\usepackage{url} 
\usepackage{hyperref}
\usepackage{subcaption}
\usepackage{algorithm}
\usepackage{algpseudocode}
\usepackage{mdframed}   
\usepackage{dsfont}     
\usepackage{soul}       

\usepackage{bbm,color}

\def\deg{d}         

\def\pof{\mbox{{\it pof}}@}

\def\pofcutoff{\mbox{{\it pofcutoff}}@}
\def\maxpof{\text{{\it maxp\kern-1pt o\kern-.25pt f}}}
\def\pof{\mbox{{\it pof}}@}

\def\Re{\operatorname{Re}}

\mathcode`\@="8000
{\catcode`\@=\active \gdef@{\mkern1mu}} 

\begin{document}

\title{Polynomial Preconditioning for Indefinite Matrices }

\author{Hayden Henson\footnotemark[2]
\and Ronald B. Morgan\footnotemark[3] }

\maketitle

\renewcommand{\thefootnote}{\fnsymbol{footnote}}
\footnotetext[2]{Department of Mathematics, Baylor University, Waco, TX 76798-7328 ({\tt Hayden\_Henson1@baylor.edu}).}
\footnotetext[3]{Department of Mathematics, Baylor University, Waco, TX 76798-7328 ({\tt Ronald\_Morgan@baylor.edu}).}
\renewcommand{\thefootnote}{\arabic{footnote}}

\begin{abstract}
Polynomial preconditioning is an important tool in solving large linear systems and eigenvalue problems.  A polynomial from GMRES can be used to precondition restarted GMRES and restarted Arnoldi.  Here we give methods for indefinite matrices that make polynomial preconditioning more generally applicable.  The new techniques include balancing the polynomial so that it produces a definite spectrum.  Then a stability approach is given that is specialized for the indefinite case.  Also, very complex spectra are examined.  Then convergence estimates are given for polynomial preconditioning of real, indefinite spectra.  Finally, tests are preformed of finding interior eigenvalues.
\end{abstract}

\begin{keywords}
 linear equations, polynomial preconditioning, GMRES, indefinite matrices
\end{keywords}

\begin{AMS}
65F10, 15A06, 
\end{AMS}

\pagestyle{myheadings}
\thispagestyle{plain}
\markboth{H. HENSON and R. B. MORGAN}{INDEFINITE POLYNOMIAL PRECONDITIONING}

\section{Introduction}

Polynomial preconditioning is a powerful tool for improving convergence when solving large linear equations~\cite{PPGStable} or finding eigenvalues~\cite{PPArn}.  However, there can be difficulties for indefinite linear equations and interior eigenvalue problems.  Here we give techniques for polynomial preconditioning to be effective in these situations. 
This is important because indefinite/interior problems tend to be difficult and so especially benefit from polynomial preconditioning.    

Polynomial preconditioning has been extensively studied; see for example~\cite{La52B,Sti58,Rutis,Sa84b,Sa87b,As87,SmSa,FiRe,AsMaOt,Jo,vGi95,Sa03,Tho06,seed,PPG,LiXiVeYaSa,PPArn,PPGStable,YeXiSa}.  The GMRES polynomial was used in~\cite{NaReTr,Jo94b} for Richardson iteration and in~\cite{Tho06,seed,PPG} for polynomial preconditioning.  However, more recently in~\cite{PPArn,PPGStable}, the GMRES polynomial was improved in efficiency, stability, and ease of determining the polynomial.  This implementation uses roots of the polynomial which enhances stability and allows for the insertion of extra copies of roots for further stability control.  

For indefinite problems, it is desirable to have the polynomial preconditioning turn the spectrum definite.  For this, we choose a polynomial that we call ``balanced".  This balanced polynomial has derivative zero at the origin.  Several methods for balancing are given.  Some guidance is included for which to use.  Also cubic Hermite splines are suggested for checking if a polynomial will be effective.  Significantly complex spectra are considered, and one conclusion is that balancing will probably not be helpful if the origin is mostly surrounded by eigenvalues.

Next, the stability control method in~\cite{PPArn,PPGStable} may be ineffective for indefinite problems, because adding small roots to one side of the origin can increase the polynomial size and variability on the other side.  This is addressed by not adding roots on the smaller side and applying stability fixes of deflating eigenvalues and a short GMRES run.  

Section 2 of the paper has quick review of items needed in this paper.  Section 3 has techniques for balancing the polynomial.  Section 4 focuses on matrices with significantly complex spectra.  Then stability of the polynomial is in Section 5, and Section 6 has convergence estimates for polynomial preconditioning indefinite linear equations.  Finally, interior eigenvalue problems are in Section 7.

\section{Review}

\subsection{Polynomial Preconditioning with the GMRES Polynomial}

Polynomial preconditioning is a way to transform the spectrum of a matrix and thus improve convergence of Krylov iterative methods.
For linear equations with polynomial $p$ and right preconditioning, this is 
\[
Ap(A)y = b, \qquad x = p(A)y. \label{eqn:ppsys}
\]
Defining $\varphi(z) \equiv z@ p(z)$, the preconditioned system of linear equations is 
\[\varphi(A)y = b.\]

In~\cite{PPArn,PPGStable}, the polynomials are found with the GMRES algorithm~\cite{SaSc}.   Starting with the GMRES residual polynomial $\pi$, the polynomial $\varphi$ is chosen as $\varphi(z) = 1 - \pi(z)$, and thus $p$ is also determined.  The roots of $\pi$ are the harmonic Ritz values~\cite{IE,PaPavdV,IEN}, and they are used to implement both polynomials $\varphi$ and $p$.  Then GMRES can also be used to solve the linear equations as a part of polynomial preconditioned GMRES, see Algorithm \ref{Alg:PPGMRESmain}, which is from~\cite{PPGStable}.  
Note that if $A$ is a real matrix, then $\varphi$ has real coefficients and $\varphi(A)$ is real.

\begin{algorithm}[h!]
    \caption{Polynomial Preconditioned GMRES, PP(d)-GMRES(m)}
\medskip
\begin{enumerate}
\item {\bf Construction of the polynomial preconditioner:} 
\begin{enumerate}
\item For a degree $d$ preconditioner, run one cycle of GMRES($\deg$) using a random starting vector.
\item Find the harmonic Ritz values $\theta_1, \ldots, \theta_\deg$, which are the roots of the GMRES polynomial: with Arnoldi decomposition $AV_\deg = V_{\deg+1} H_{\deg+1,\deg}$, find the eigenvalues of $H_{\deg,\deg}^{} + h_{\deg+1,\deg}^2 @@f e_\deg^T$, where $f = H_{\deg,\deg}^{-*} e_\deg^{}$ with elementary coordinate vector $e_\deg=[0,\ldots,0,1]^T$.     
\item Order the GMRES roots using Leja ordering~\cite[Alg.~3.1]{BaHuRe} and apply stability control as in~\cite[Algorithm 2]{PPGStable}.
\end{enumerate}
\medskip
\item {\bf PP-GMRES:} Apply restarted GMRES to the matrix $\varphi(A) = I-\Pi_{i=1}^{d}(I - A/\theta_i)$ to compute an approximate solution to the right-preconditioned system $\varphi(A)y=b$, using~\cite[Algorithm 1]{PPGStable} for $\varphi(A)$ times a vector. To find $x$, compute $p(A)y$ with~\cite[Algorithm 3]{PPGStable}.
\end{enumerate}
\label{Alg:PPGMRESmain}
\end{algorithm}

\subsection{Stability control}

For a matrix with an eigenvalue that stands out from the rest of the spectrum, the GMRES residual polynomial $\pi$ generally has a root at the eigenvalue.  The slope of this polynomial is likely to be very steep at that root which can lead to ill-conditioning and cause $\varphi(A)$ times a vector to be unstable.  To improve stability, extra copies of roots corresponding to outstanding eigenvalues can be added to $\pi$.
This is implemented in~\cite[Algorithm 2]{PPGStable} (see also~\cite[p.~A21]{PPArn}).  For each root $\theta_k$, one computes a diagnostic quantity called $\pof(k)$ that measures the magnitude of $\pi(\theta_k)$ with the $(1-z/\theta_k)$ term removed.  When $\log_{10}(\pof(k))$ exceeds some threshold $\pofcutoff$, extra $(1-z/\theta_k)$ terms are appended to $\pi$ (in~\cite[Algorithm 2]{PPGStable}, $\pofcutoff$ is set to 4). 

\subsection{Range Restricted GMRES}

Some image processing problems have matrices with zero and very small eigenvalues that correspond to noise.  The associated linear equations need to be solved so that the effect of these small eigenvalues is essentially removed.  Range Restricted GMRES (RRGMRES)~\cite{CaLeRe} chooses $A\cdot b$ as the initial vector for its subspace.  This starting vector has small eigencomponents corresponding to the small eigenvalues, thus reducing their effect.

\section{Balancing the Polynomial}

In this section, we give ways of adjusting polynomial preconditioning to make it more effective for indefinite problems.  We begin with an example showing that polynomial preconditioning with the GMRES polynomial can be very effective but needs some improvement. 

{\it Example 1.} We use a matrix that is bidiagonal with diagonal elements $-2500,$ $-2499, -2498, \ldots, -2, -1, 1, 2, 3, \ldots, 2499, 2500$ and super-diagonal elements all 1.  So while the matrix is nonsymmetric, the spectrum is real and is mirrored about the origin.  The right-hand side is generated random normal and then is scaled to norm 1.  The residual norm tolerance is $10^{-10}$.

Table 3.1 has results comparing restarted GMRES(50) to PP(d)-GMRES(50), which stands for polynomial preconditioned GMRES with $\varphi$ polynomial of degree $d$ and GMRES restarted at dimension 50.  The preconditioning is effective for this very indefinite case.  Comparing no polynomial preconditioning ($d=1$) to $d = 50$, the polynomial preconditioning improves the time by a factor of 200.  The top of Figure 3.1 shows the polynomials used for degrees 5, 25 and 50.  The spectra are significantly improved by the polynomial with degrees 25 and 50.  The eigenvalues are mapped to near 1, except for eigenvalues near the origin.  This explains why the method is effective for these polynomials.  However, with the degree 5 polynomial, many eigenvalues near zero are mapped very close to zero.  This is especially shown in the closeup in the bottom half of the figure, with the 200 eigenvalues closest to the origin all mapped very near zero.

Looking further at the results in Table 3.1, higher degree polynomials can be even more effective.  However for degree 150, it takes about 10 times longer than degree 151.  This is explained by the closeup in the bottom of Figure 3.2. Both degree 150 and 151 polynomials dip into the lower half of the plane, so the spectrum of $\varphi(A)$ is indefinite. Both have 11 negative eigenvalues.  But the important difference is that degree 150 happens to have an eigenvalue fall very near zero.  This smallest eigenvalue of $\varphi(A)$ is at $-2.6*10^{-4}$, compared to smallest at $3.0*10^{-3}$ for degree 151.  Having one very small eigenvalue can significantly slow down GMRES(50).   

So this example gives two reasons for why we need a polynomial that stays in the top half of the plane over the spectrum.  Such a polynomial gives a definite spectrum and avoids creating very small eigenvalues for $\varphi(A)$.

\begin{table}
\caption{Bidiagonal matrix from Example 1 which is nonsymmetric with a real, symmetric (mirrored) spectrum.  PP(d)-GMRES(50) is used to solve linear equations.   MVP's is the total matrix-vector products, v op's is the total length-n vector operations and dp's is the total dot products. }
\begin{center}
\begin{tabular}{|c||c|c|c|c||c|c|c|c|}  \hline\hline
& \multicolumn{4}{|c||}{GMRES Polynomial}       &  \multicolumn{4}{|c|}{Balanced Poly }  \\ 
& \multicolumn{4}{|c||}{}                       &  \multicolumn{4}{|c|}{- Method 1 - Added Root }  \\   \hline
d, deg                 & Time              & MVP's     &  v op's   & dp's      & Time      & MVP's     &  v op's   & dp's  \\ 
of $\varphi$           & (sec's)           & (tho's)  &  (mil's)   & (mil's)   &  (sec's)  & (tho's)   & (mil's)   & (mil's) \\   \hline \hline
1                       & 2346              & 4363      &  245      & 116       &           &           &       &   \\ 
(No PP)                 &                   &           &           &           &           &           &       &    \\   \hline
5                       &  -                &  -        &   -       & -         &  198      & 2014      & 20.5  & 8.91   \\ \hline
10                      & 280               & 4305      &  28.0     & 11.4      &  47.1     & 767       & 4.61  & 1.85  \\ \hline 
50                      & 11.5              & 444       &  0.94     & 0.24      &  2.77     & 95.3      & 0.20   & 0.051 \\ \hline 
100                     & 10.5              & 535       &  0.84     & 0.15      &  2.18     & 86.7      & 0.14   & 0.028 \\ \hline 
150                     & 43.6              & 2442      &  3.36     & 0.44      &  1.79     & 64.9      & 0.11   & 0.023 \\ \hline
151                     & 4.35              & 212       &  0.31     & 0.048     &  1.86     & 68.6      & 0.12   & 0.024 \\ \hline 
200                     & 3.85              & 180       &  0.27     & 0.044     &  1.86     & 61.5      & 0.12   & 0.028 \\ \hline 
250                     & 3.60              & 146       &  0.24     & 0.047     &  2.07     & 47.9      & 0.12   & 0.036   \\ \hline 
\hline
\end{tabular}
\end{center}
\end{table}

\begin{figure}
\vspace{-2.6in}
\hspace{-.9in}
\includegraphics[scale=.75]{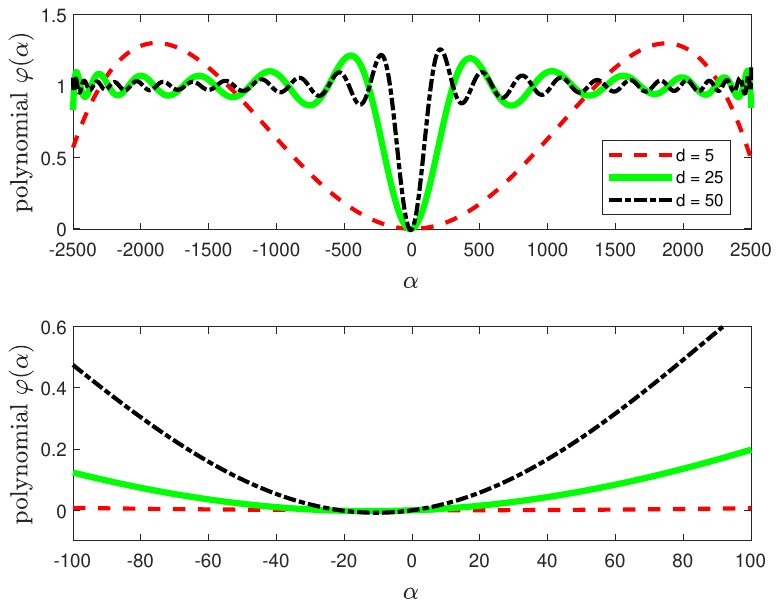}
\vspace{-2.8in}
\caption{Bidiagonal matrix of size $n = 5000$.  Top has $\varphi$ polynomials of degree $d=5$, $d=25$ and $d=50$.  Bottom has a closeup of the same polynomials. }
\label{fig:SymmSpec}
\vspace{-0.1in}
\end{figure}


\begin{figure}
\vspace{-2.6in}
\hspace{-.9in}
\includegraphics[scale=.75]{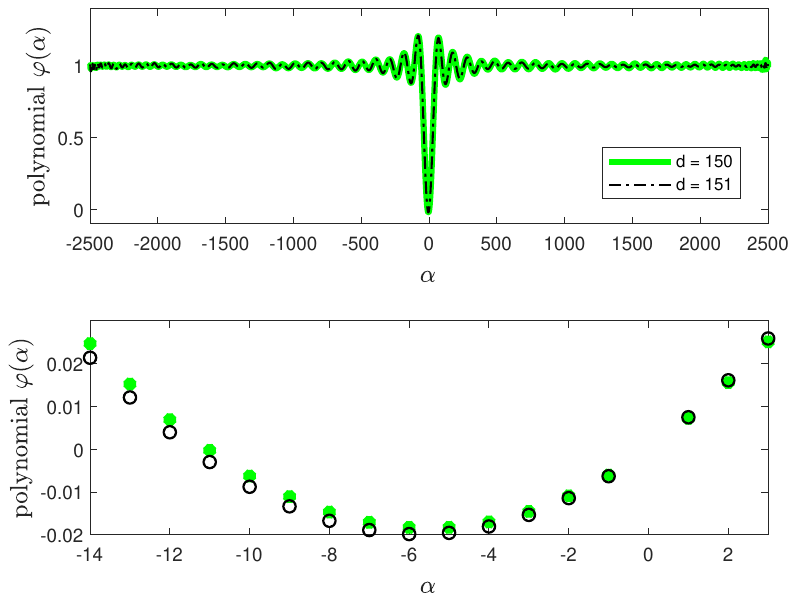}
\vspace{-2.8in}
\caption{Bidiagonal matrix of size $n = 5000$.  Top has $\varphi$ polynomials of degree $d=150$ and $d= 151$.  Bottom has a closeup of the same polynomials plotted at the eigenvalues. }
\label{fig:SymmSpec2}
\end{figure}

We call a polynomial ``balanced" if it has slope zero with respect to the real axis at the origin.  This way it will be positive over the real axis near the origin and hopefully over all the spectrum.  We will give several ways of balancing the polynomial. 

\subsection{Balance Method 1:  Add a root}

The first way of balancing the polynomial is to add a root that makes the slope zero at the origin.  As with adding roots for stability in Subsection 2.2, the resulting polynomial is no longer a minimum residual polynomial. But it may nearly have that property since it is a modification of the GMRES polynomial.

The slope of the polynomial $\varphi$ at the origin is 
\begin{equation} \label{eq:3.1}
    \varphi'(0) = \sum_{i=1}^d\frac{1}{\theta_i} \nonumber.
\end{equation}

In order to balance it, we add one extra root to $\pi$ to make the slope $\varphi$ at the origin zero.  We call this root the ``balancing root" defined as $\eta=-1/\sum_{i=1}^{d}\frac{1}{\theta_i}$ and denote the new polynomial as $\varphi_1$, see Algorithm~\ref{alg:one}.

\begin{algorithm}
\caption{Balance Method 1: Add a Root.}\label{alg:one}
\begin{description}
 \item[0.] Apply GMRES(d) to form a polynomial $\varphi(\alpha)$ of degree $d$.  Let the roots of the corresponding polynomial $\pi(\alpha)$ (harmonic Ritz values) be $\theta_1, \ldots, \theta_d$.
 \item[1.] Compute $\eta = - 1/\sum_{i=1}^d\frac{1}{\theta_i}$ and add it to the list of polynomial roots of $\pi$.
 \item[2.] The new polynomial of degree $d+1$ is $\varphi_1(\alpha) = 1-\pi(\alpha)\cdot\Big(1 - \frac{\alpha}{\eta}\Big)$.
\end{description}
\end{algorithm}

{\it Example 1 (continued).} 
The right half of Table 3.1 is for the balanced polynomial.  All of the $\varphi_1$ polynomials used are one higher degree than the degree listed on the left because of the added root.  The results are consistently better with the balancing.  Even the degree six polynomial ($d=5$ before the added root) is effective, converging in 198 seconds compared to over 2000 seconds without polynomial preconditioning.  For polynomials with $d=50$ and above, the times are all below three seconds, and dot products are reduced by more than three orders of magnitude versus no polynomial preconditioning.

{\it Example 2.}  We next use a matrix that is bidiagonal with diagonal elements $-100, -99,$ $-98, \ldots, -2, -1, 1, 2, 3, \ldots, 4899, 4900$ and super-diagonal elements all 1.  This matrix is much less indefinite than the first one. 
Table 3.2 has timing results for PP(d)-GMRES(50) with different degree polynomials and with different balance methods.  This example motivates the balance methods that will be given in the next few subsections.  

With Balance Method 1, the best times are an improvement of more than a factor of 150 compared to no preconditioning. 
However, for the degree five polynomial, balancing is not necessary.  Figure 3.3 shows why.  The unbalanced approach has a polynomial that goes negative over the negative part of the spectrum of $A$, so the resulting polynomial preconditioned spectrum remains indefinite.  But it does not have as many small eigenvalues as with Balance Method 1 which is fairly flat around the origin.  Specifically, with the original degree five polynomial, $\varphi(A)$ has 100 negative eigenvalues, but only 12 with absolute value less than 0.02 and smallest $3.2*10^{-3}$.  Meanwhile, with the balanced degree six polynomial, $\varphi_1(A)$ has all positive eigenvalues, but has 108 less than 0.02 and the smallest is $6.7*10^{-6}$.  
At degree 10, convergence is not achieved without balancing as the residual norm stalls at $0.0217$.  Balance Method 1 is a big improvement for this degree and degrees 15 and 25. For degrees 50 and higher, convergence is reached in under three seconds even without balancing, though balancing still yields improvements in some cases. 

The degree 9 polynomial with Balance Method 1 (so degree 10 with the extra root) does poorly.  The added root is around -645 and the linear factor $\left(1+\frac{\alpha}{645}\right)$  causes the polynomial to increase in size at the large part of the spectrum.  Figure~\ref{fig:Spec100neg2} shows this polynomial dips below zero and so gives an indefinite spectrum.  The degree 200 polynomial has a similar problem as degree 9.  Balance Method 2 is developed next to deal with these situations.

\begin{table}
\vspace{-0.1in}
\caption{Bidiagonal matrix with diagonal $-100, -99, \ldots, -1, 1, 2, 3, \ldots, 4900$, and superdiagonal of ones.  PP(d)-GMRES(50) is used to solve linear equations.  Times are compared for different ways of balancing the polynomial.  Times are given in seconds.  For Balance 2, ``same" means that no root is subtracted but a root is added, so it is the same as Balance 1.  Superscript~$^*$ indicates the residual norm only reaches $7.94*10^{-10}$, and superscript~$^\dagger$ indicates the residual norm only reaches $6.44*10^{-10}$.  For Balance 4, the degree 10 polynomial is composed of a degree 5 inner polynomial and a degree 2 outer polynomial.  Similarly, degree 15 and 25 have inner polynomial of degree 5, while the higher ones have inner polynomial of degree 10.}
\begin{center}
\begin{tabular}{|c||c|c|c|c|c|}  \hline\hline

d - degree              & No balance        & Balance 1     &  Balance 2        & Balance 3             & Balance 4 \\ 
of poly $\varphi$       & (seconds)         &  add root    &  remove, add       & RRG poly              & Composite  \\   \hline \hline
No PP                   & 229               &               &                   &                       &             \\ \hline
5                       & 18.6              & 52.8          & 118               & 97.4                  &              \\ \hline
9                       & 17.3              & 373           & 28.2              & 13.9                  &               \\ \hline 
10                      & -                 & 8.44          & same              & 9.48                  &  34.9          \\ \hline 
15                      & 10.6              & 3.16          & 3.09              & 3.33                  &  19.2           \\ \hline 
25                      & 20.4              & 2.33          & 2.62              & 1.92                  &  9.28            \\ \hline 
50                      & 2.92              & 1.83          & 2.11              & 1.59                  &  1.53             \\ \hline
100                     & 1.95              & 1.92          & same              & 1.23                  &  1.06              \\ \hline 
150                     & 1.82              & 1.44          & 3.04              & $1.87^*$              &  1.10               \\ \hline 
200                     & 2.08              & -             & 2.69              & $2.20^\dagger$        &  0.77                \\ \hline 
\hline
\end{tabular}
\end{center}
\label{Tab:Ex2}
\end{table}

\begin{figure}
\vspace{-2.6in}
\hspace{-.9in}
\includegraphics[scale=.75]{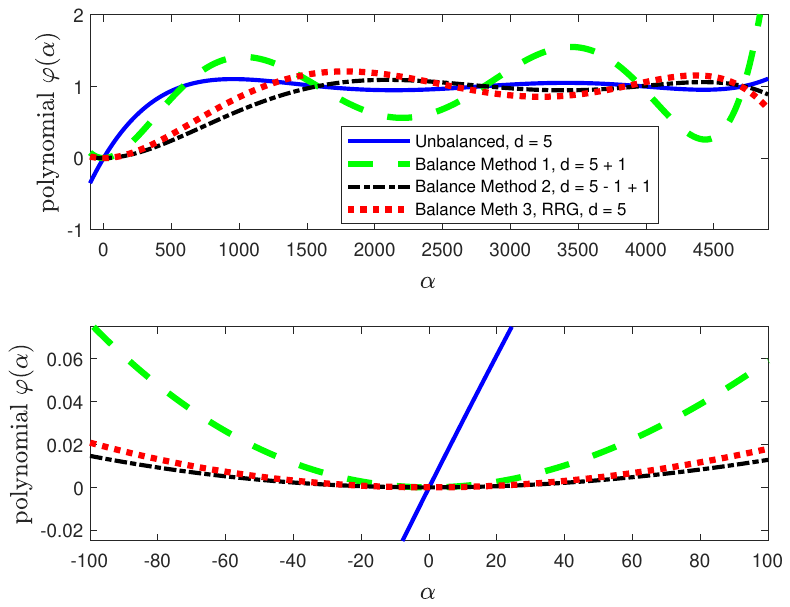}
\vspace{-2.8in}
\caption{Bidiagonal matrix of size $n = 5000$ with 100 negative eigenvalues.  Top has $\varphi$ polynomials of original degree $d=5$ with different balancing methods.   Bottom has a closeup of the same polynomials near the origin. }
\label{fig:Spec100neg}
\vspace{-0.1in}
\end{figure}

\begin{figure}
\vspace{-2.6in}
\hspace{-.9in}
\includegraphics[scale=.75]{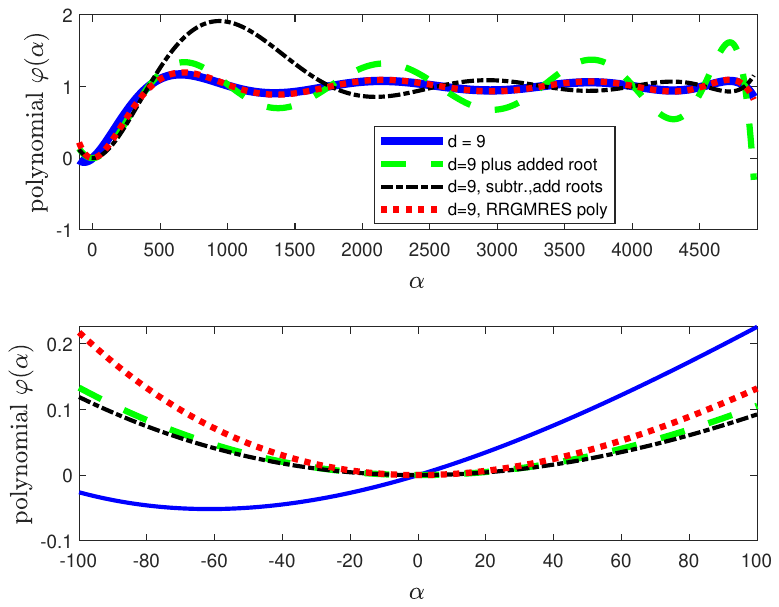}
\vspace{-2.8in}
\caption{Bidiagonal matrix of size $n = 5000$ with 100 negative eigenvalues.  Top has $\varphi$ polynomials of degree $9$ with different balancing.   Bottom has a closeup of the same polynomials near the origin. }
\label{fig:Spec100neg2}
\vspace{-0.2in}
\end{figure}

\subsection{Balance Method 2:  Remove root(s), then add a root}

Sometimes it may be beneficial to remove one or two roots from the GMRES polynomial $\pi$, but still add a balancing root. The motivation for first removing a root is to decrease the value of $|\varphi'(0)|$.  This gives a balancing root $\eta$ further from the origin and linear factor $(1-\frac{\alpha}{\eta})$ closer to 1 across the spectrum of $A$.  To make the value of $|\varphi'(0)|$ smaller, we look for the harmonic Ritz value whose reciprocal (or sum of reciprocals in the complex case) is closest to $\varphi'(0)$.  Balance Method 2 is in Algorithm~\ref{alg:bal2}.

\begin{algorithm}
\caption{Balance Method 2: Remove root(s), then add a root}\label{alg:bal2}
\begin{description}
 \item[0.] Apply GMRES(d) to form a polynomial $\varphi(\alpha)$ of degree $d$.  Let the roots of the corresponding polynomial $\pi(\alpha)$ (harmonic Ritz values) be $\theta_1, \ldots, \theta_d$.
 \item[1.] Compute the difference from the derivative:
 \begin{itemize}
    \item For each real root $\theta_i$, compute $|\varphi'(0)-\frac{1}{\theta_i}|$
    \item For complex conjugate pairs $(\theta_i,\overline{\theta_{i}})$, compute $|\varphi'(0)-(\frac{1}{\theta_i}+\frac{1}{\overline{\theta_{i}}})|$
 \end{itemize}
 \item[2.] Let $\xi$ be the inverse (or sum of inverses) with the smallest difference from $\varphi'(0)$.
 \item [3.] If $|\varphi'(0)-\xi|\geq|\varphi'(0)|$, do not remove any roots and add $\eta=- 1/\sum_{i=1}^d\frac{1}{\theta_i}$ to the list of polynomial roots of $\pi$.
 \item [4.] If $|\varphi'(0)-\xi|<|\varphi'(0)|$, remove the root(s) whose sum yields $\xi$ and add \\$\eta=- 1/\left(\sum_{i=1}^d\frac{1}{\theta_i}-\xi\right)$ to the list of polynomial roots of $\pi$.
 \item[5.] The new polynomial of degree $d$, $d+1$, or $d-1$ is $\varphi_2=1-\pi_r(\alpha)\cdot\left(1-\frac{\alpha}{\eta}\right)$. Where $\pi_r$ is the $\pi$ polynomial after having root(s) removed
\end{description}
\end{algorithm}

{\it Example 3.} 
We consider another bidiagonal matrix with diagonal elements $-100,$ $-99,\ldots,-2,-1,1,2,\ldots, 9850, 9860,9870,\dots,10330,10340,10350$ and super-diagonal of all ones. As can be seen in Table \ref{table:Bidiag3}, polynomial preconditioning provides improvement of up to $14$ times compared to no preconditioning and up to $43$ times improvement with Balance Method 2. For the degree 10 polynomial, we see that Balance Method 2 harms convergence compared to Balance Method 1 and to no balancing. Caution should be used when applying Balance Method 2 to low degree polynomials, as removing a root relatively takes away a lot of information.  Next, with a degree 40 polynomial, the $\varphi(A)$ has an eigenvalue at $2.5* 10^{-4}$ which slows down convergence considerably. Balance Method 1 helps this problem (see top right of figure \ref{fig:Bal2_40}), but still produces a polynomial that is negative over the intervals $(9900,10090),(10230,10280), \textrm{ and } (10340,10350)$ which can be seen in the top left of Figure \ref{fig:Bal2_40}.  This is in contrast to Balance Method 2 which gives a polynomial which remains positive definite over the spectrum of $A$ (see bottom of Figure \ref{fig:Bal2_40}) and gives much faster convergence.

\subsection{Using Cubic Splines}

Since the goal of the balancing root $\eta$ is to ensure that the $\varphi$ polynomial remains positive over the spectrum of $A$, it is useful to verify that the GMRES polynomial $\pi$ stays below $1$.  Then $\varphi$ will stay above zero.   
For this, we employ Hermite cubic splines to estimate the values of the balanced $\pi$ over the intervals between its real roots. This approach is a test to determine the polynomial’s suitability for preconditioning.
We develop cubic splines $C_j$ in each interval $(\theta_j,\theta_{j+1})$ which meet the following conditions:
\begin{enumerate}
    \item $C_j(\theta_j)=C_j(\theta_{j+1})=0$
    \item $C'(\theta_j)=\pi'(\theta_j)$ and $C_j'(\theta_j+1)=\pi'(\theta_j+1)$.
\end{enumerate}
The goal is to see if $C_j$ stays below 1 over the interval, which suggests that $\pi$ behaves likewise. 
We do need not check the intervals where $\pi'(\theta_j)<0$, because it is only possible for $C_j(x)\geq 0$ over $[\theta_j,\theta_{j+1}]$ if $\pi'(\theta_j)\geq 0$.  

\begin{algorithm}
    \caption{Using cubic splines to show positive definiteness}
    \begin{description}
        \item[0.] Let $\theta_1\leq\theta_2\leq\cdots\leq\theta_{\Tilde{d}}$ be the real harmonic Ritz values which are the roots of $\pi$, and $\Tilde{\theta}_{min},\Tilde{\theta}_{max}$. be the Ritz values with smallest (most negative) and largest real parts respectively. 
        
        \hspace{-1cm} For each interval $(\theta_j,\theta_{j+1}), \textrm{ } j=1,\dots, \Tilde{d}-1$, where $\pi'(\theta_j)\geq 0$ and where $\pi'(\theta_j)$ and $\pi'(\theta_{j+1})$ are not both 0, do:
        \item[1.] Calculate $C_j(x)=\frac{1}{6}a(x-\theta_j)^3+\frac{1}{2}b(x-\theta_j)^2+c(x-\theta_j),$
        where $a=\frac{6(\pi'(\theta_{j+1})+\pi'(\theta_j))}{(\theta_{j+1}-\theta_j)^2}$, $b=\frac{-(2\pi'(\theta_{j+1})+4\pi'(\theta_j))}{(\theta_{j+1}-\theta_j)}$, and $c=\pi'(\theta_j)$. 
        \item[2.] Find the critical point(s) 
        $\hat{x}_j=\theta_j+\frac{-b\pm\sqrt{b^2-2ac}}{a}.$ 
        Select the root $\hat{x_j}\in(\theta_j,\theta_{j+1})$.
        \item[3.] If $C_j(\hat{x}_j)>1$ and any of the following hold:
        \begin{itemize}
            \item $(\hat{x}_j\leq\Re(\Tilde{\theta}_{min})\leq\theta_{j+1} \textrm{ and } C_j(\Re(\Tilde{\theta}_{min}))>1)$
            \item $(\theta_j\leq\Re(\Tilde{\theta}_{max})\leq\hat{x}_j \textrm{ and } C_j(\Re(\Tilde{\theta}_{max}))>1)$
            \item $\Re(\Tilde{\theta}_{min})\notin (\hat{x}_j,\theta_{j+1}) \textrm{ and } \Re(\Tilde{\theta}_{max})\notin(\theta_j,\hat{x}_j))$
        \end{itemize}
        Then $\varphi$ is not positive over the spectrum of $A$.    \end{description}
\end{algorithm}

The last part of the algorithm considers the Ritz values along with the harmonic Ritz values.  
Spurious harmonic Ritz values can occur outside of the spectrum, as is also discussed in Section 5.

{\it Example 3. (cont.)}  Figure \ref{fig:Bal2_40} shows that the Balance Method 1 polynomial of degree 40 has large fluctuation over the spectrum of $A$.  The cubic spline test in Algorithm 4 appropriately flags the interval between the two harmonic Ritz values at about $9990$ and $10{,}115$, so the polynomial is rejected.

Another, more complicated, option for testing a polynomial is to actually find the maximum value of $\pi$ on each interval between real roots.  This could be done with a bisection method using the polynomial and its derivative.

\begin{table}\caption{\label{table:Bidiag3} Bidiagonal matrix of size $n=10,000$ with diagonal $-100,-99,\dots,$ $-2,-1,$ $1,2,\dots, 9850, 9860,9870,\dots,10330,10340,10350$ and super diagonal of ones. PP(d)-GMRES(50) is used to solve linear equations. - denotes polynomial degrees where Balance Method 1 resulted in an indefinite spectrum for $\varphi(A)$}. 
\begin{center}
\begin{tabular}{|c||c|c|c|c|c|c|}  \hline\hline

d - degree              & \multicolumn{2}{c|}{No balance}        & \multicolumn{2}{c|}{Balance 1}     & \multicolumn{2}{c|}{Balance 2}  \\ 
of poly $\varphi$          & \multicolumn{2}{c|}{}                  & \multicolumn{2}{c|}{add root}      & \multicolumn{2}{c|}{subtr., add}   \\ \hline \hline
                        & MVP               & Time               & MVP       & Time                   & MVP       & Time          \\
                        & (thous)           & (sec)              & (thous)   & (sec)                  & (thous)   & (sec)          \\ \hline 
No poly                 & 872               & 976                &          &                         &           &                 \\ \hline
10                      & 210               & 28.1               & 298       & 40.4                   & 798       & 136              \\ \hline
40                      & 12,940            & 749                & 4,394     & 254                    & 42.6      & 3.6               \\ \hline
70                      & 69.3              & 5.6                & 377       & 19.0                   & 26.0      & 2.2                \\ \hline
100                     & 59.0              & 4.7                & 203       & 10.6                   & 20.2      & 1.9                 \\ \hline
\hline
\end{tabular}
\end{center}
\end{table}

\begin{figure}[t!]
\begin{center}
  \includegraphics[scale=.60]{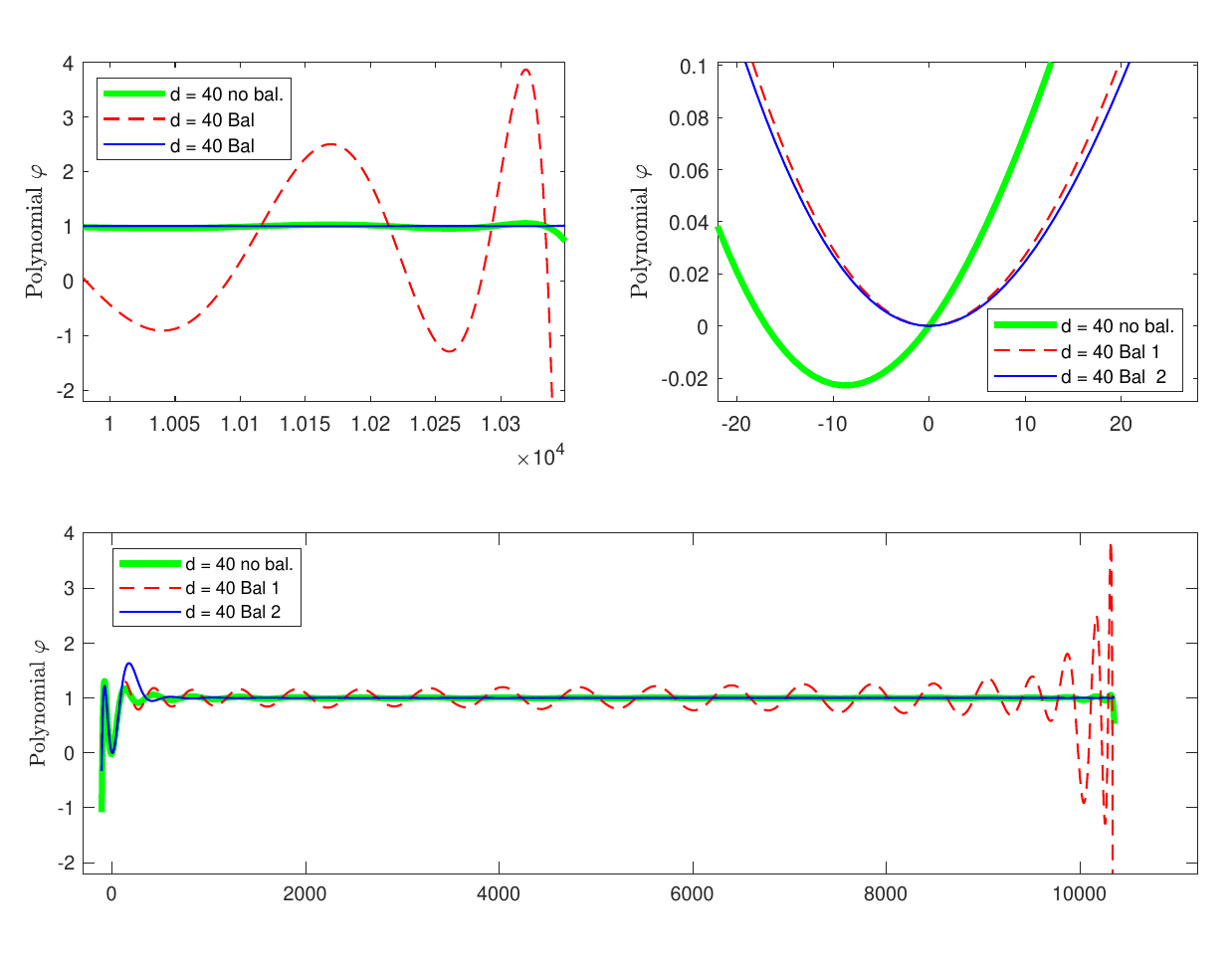}
\end{center}
\vspace{-.27in}
\caption{\label{fig:Bal2_40} 
Bidiagonal matrix of size $n=10{,}000$. The bottom graph has the $\varphi$ polynomial of degree 40 with no balance, balance method 1 and balance method 2. Top right has the same polynomials zoomed in at the origin. Top left demonstrates the intervals where the polynomial from Balance Method 1 is negative.} 
\end{figure}

\subsection{Balancing with a polynomial from Range Restricted GMRES}

This section discusses two methods that use Range Restricted GMRES~\cite{CaLeRe} (RRGMRES).  This approach automatically gives a balanced polynomial.  There is more than one way to represent the polynomial from RRGMRES.  Here we implement it in Newton form and call it Balance Method 3; see Algorithm~\ref{alg:bal3}.


\begin{algorithm}
\caption{Determine the Polynomial from Range Restricted GMRES for Balance Method 3}\label{alg:bal3}
\begin{description}
 \item[0.]  Choose the degree $d$ of the $\varphi$ polynomial.
 \item[1.] Apply the Arnoldi iteration with starting vector $A\cdot b$ for $d$ iterations.  Compute Ritz values $\theta_1, \ldots, \theta_d$ (regular Ritz, not harmonic).
 \item[2.] Generate a matrix $V$ with Newton basis vectors as columns.  So the vectors are $b, (A-\theta_1)b, (A-\theta_2)(A-\theta_1)b, \ldots, (A-\theta_{d-1}) \cdots (A-\theta_2)(A-\theta_1)b $.  This can be modified to keep real vectors for the complex Ritz vector case.  
 \item[3.] Solve the Normal Equations, so $(AV)^T (AV)g = (AV)^T b$.  Then the Newton coefficients are in $g$.
\end{description}
\end{algorithm}


Determining the polynomial for Balance Method 3 is not efficient for high degree polynomials as it uses double the matrix-vector products.  Also, it may not always be stable due to the non-orthogonal basis for $V$ and the use of Normal Equations (the columns of $V$ do approximate an orthogonal matrix).
For a more stable and more complicated implementation of this Newton form of the polynomial, see~\cite[Subsection 3.3]{PPG} (this can be adjusted for RRGMRES).

Balance Method 4 uses a composite polynomial with the Newton form of the RRGMRES polynomial as the inner polynomial.  The outer polynomial comes from polynomial preconditioned GMRES. For details of using composite polynomials, we refer to~\cite{PPGStable} where they are called double polynomials.

{\it Example 2. (cont.)}  The two right-most columns have results with Balance Methods 3 and 4.  These methods are mostly competitive and Method 4 has the best times for high degree polynomials.  Note that Balance Method 3 does not quite give full accuracy at the high degrees.

\subsection{Choice of balancing}

We first give an example where balancing is detrimental, then discuss when balancing should be used and which version is appropriate.  
 
{\it Example 4.}  Let the matrix be diagonal of size $n=5000$ with eigenvalues $-1000, -999, -998, \ldots, -100, 0.1, 0.2, 0.3, \ldots 1, 2, 3, \ldots 4090$.  We apply polynomial preconditioning with and without balancing.  PP(50)-GMRES(50) converges to residual norm below $10^{-10}$ in 10 cycles without balancing.  Using balancing is disastrous.  With Method 1, it takes 4235 cycles, Balance 2 has 6713 cycles and Balance 3 uses 4185.  The top of Figure~\ref{fig:balbad} has the polynomial with and without Balance Method 1.  The unbalanced polynomial comes through the origin with significant slope that allows it to separate the small eigenvalues from the origin more than the balanced polynomial.  The smallest eigenvalue of $\varphi(A)$ without balancing is $6.12*10^{-4}$, while the smallest with balancing is much smaller at $1.27*10^{-6}$.  The bottom of Figure~\ref{fig:balbad} shows the polynomial at these small eigenvalues.

\begin{figure}[t!]
\begin{center}
\includegraphics[width=5.25in]{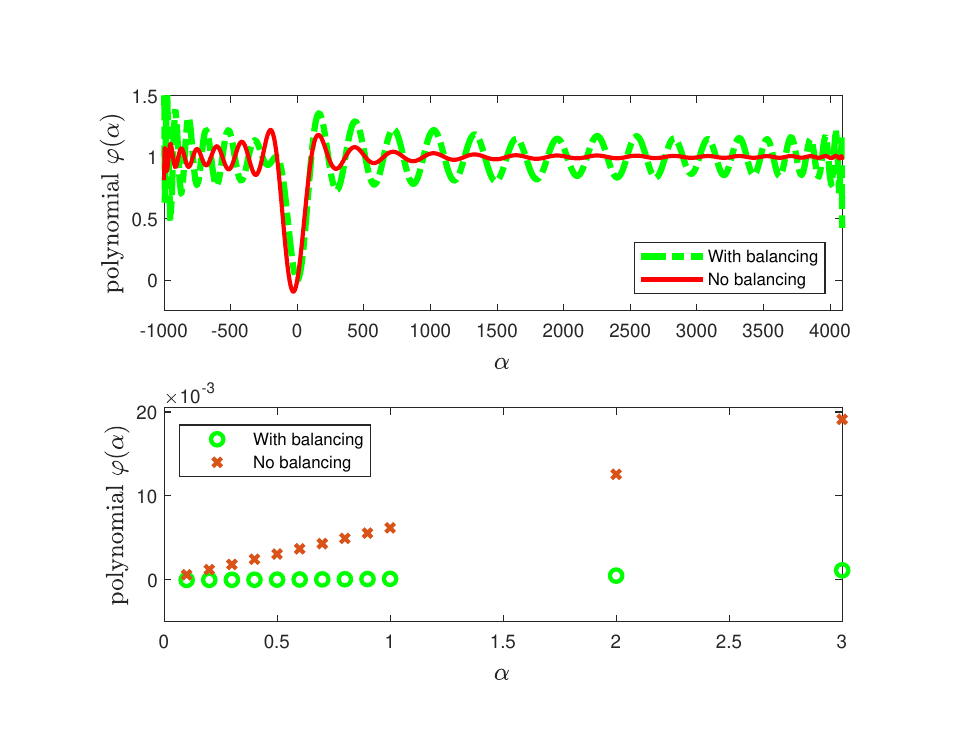}
\end{center}
\vspace*{-25pt}
\caption{\label{fig:balbad}   Matrix with a gap in the spectrum on one side of the origin.  Balance Method 1 is compared to no balancing with $d=50$.  The lower half of the figure has `x' and `o' marks showing the polynomial values at the small eigenvalues.  }
\vspace{-0.15in}
\end{figure}

Now we attempt to give guidance of when to balance.  Some rough knowledge about the spectrum is required.  Otherwise it may be necessary to experiment with and without balancing to determine which works better.  
If the eigenvalues are real or almost all real, and they are fairly equally distributed near the origin, then there will probably be a gain from balancing an indefinite spectrum and hopefully turning it definite.  For such a spectrum that is nearly mirrored on both sides of the origin, use Balance Method 1.  Otherwise, one can try Methods 1 and 2, possibly testing with splines to see if they pass.  If not, or if they don't work in an actual test, switch to Balance 3 or to Balance 4 for high degree polynomials.

As the last example showed, if there is a gap between the origin and the eigenvalues on one side of the spectrum, then it may be best to not balance.  Then the polynomial can dip down in this gap.  The preconditioned spectrum may still be definite.  
Finally, when there are eigenvalues surrounding the origin with significant imaginary parts, this is both a difficult situation and probably one where balancing is not beneficial.  This is discussed in the next section.

\section{Complex Spectra}

Matrices with spectra spread out in the complex plane need to be investigated.  A polynomial cannot be expected to be effective if the origin is completely surrounded by eigenvalues, based on the minimum modulus principle in complex analysis.  For instance, if the spectrum is on a circle centered at the origin, we would want the $\varphi$ polynomial to be zero at the origin and then on the circle have real part near one and imaginary part near zero.  However such a polynomial would have minimum modulus in the interior of the circle, violating the principle.

We look at approaching this difficult situation of having the origin surrounded in this and later sections.
We first have an example where both polynomial preconditioning and balancing are effective for a matrix with a significantly complex spectrum, but not complex at the origin.  Then the next example looks at how far we can surround the origin and still solve linear equations.

{\it Example 5.}  
We consider a Hatano-Nelson~\cite{HaNe} matrix of size $n=2500$, which has a complex and indefinite spectrum; see the lower right part of Figure \ref{fig:Hatano_Nelson}.  The results are in Table~\ref{tab:Hatano_Nelson}.  GMRES(50) does not converge; the residual norm stalls at $||r||=0.359$.  This can be fixed with polynomial preconditioning. 
For degree $d=15$, Balance Methods 3 and 4 convergence rapidly.  Balance methods 1 and 2 fail to converge, but for degrees 20 and 25 they work well.  
The left side of Figure~\ref{fig:Hatano_Nelson} shows the spectra after polynomial preconditioning of degree 25 using no balancing, and Balance Methods 1 and 2.  With no balancing, the spectrum is very indefinite.  A close-up with the two balancings in the upper right part of the figure shows the spectra are now definite.  


\begin{table}
\caption{PPGMRES applied to the Hatano-Nelson ~\cite{HaNe} with $n=2500$, $\gamma = 0.5$,  and $d =0.9*4*\textrm{rand}(n,1)$. For balance method 4, the inner polynomial is selected to have the highest possible degree, up to 15, that evenly divides the degree of the overall composite polynomial.  Times are given in seconds.}
\label{tab:Hatano_Nelson}
\begin{tabular}{|c||c|c|c|c|c|}  \hline\hline

d - degree      & No balance        & Balance 1        & Balance 2      & Balance 3     & Balance 4  \\ 
of poly $\varphi$  &             & add root         & subtr., add    & RRG poly      & Composite      \\ \hline \hline
15              & 4.66              & -                & -              & 2.08          & 2.85            \\ \hline
20              & 22.6              & 4.50             & 4.29           & 2.69          & 3.73             \\ \hline
25              & 12.0              & 3.92             & 3.92           & 3.32          & 4.03              \\ \hline
50              & 10.3              & 7.18             & -              & 2.01          & 3.70               \\ \hline
100             & 64.2              & 5.02             & same as bal 1  & 1.63          & 5.62                 \\ \hline

\end{tabular}
\end{table}

\begin{figure}[t!]
\begin{center}
\includegraphics[scale=.55]{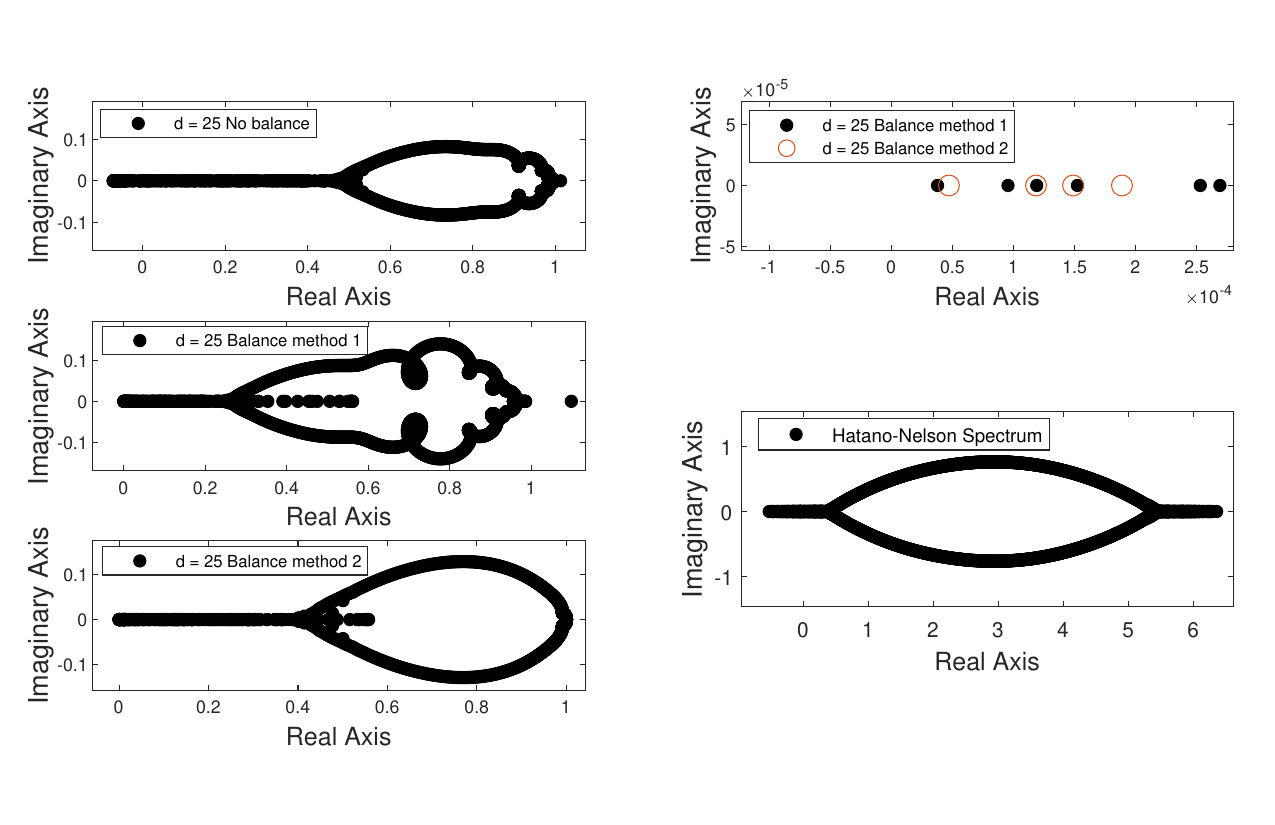}
\end{center}
\
\vspace*{-32pt}
\caption{\label{fig:Hatano_Nelson} {Graphs for the Hatano-Nelson matrix with $n=2500, \gamma=0.5$ and $d=0.9*4*\textrm{rand}(n,1)$. The bottom right graph is the original spectrum of the matrix. The top left is the spectrum transformed with a degree 25 precondition polynomial without balancing. The middle left and bottom left graphs are the spectrum transformed with the balanced method 1 and balance method 2 polynomial, respectively. The top right shows balanced methods 1 and 2 zoomed in at the origin to show that they are positive definite spectra.}}
\vspace{-0.2in}
\end{figure}


{\it Example 6.} 
Polynomial preconditioning is examined for matrices with very complex spectra.  Each matrix has size $n=2000$.  Twenty eigenvalues are equally spaced on 50 rays that move out from the origin and go much of the way around it; see the (blue) dots on the left parts of Figure~\ref{fig:complex}.  GMRES(50) is run with and without polynomial preconditioning to relative residual norm tolerance of $10^{-8}$ and with right-hand sides generated random normal and then normed to one.  No balancing is used.  The red stars on the left parts of Figure~\ref{fig:complex} are the roots of the GMRES residual polynomials of degree 50.  Plots on the right show how the spectrum is changed by the preconditioning.  The first case, shown in the top half of the figure, has eigenvalues 230 degrees around the origin.  The polynomial succeeds in transforming the original indefinite spectrum on the left into one that is nearly definite on the right (there are only six eigenvalues with negative real part after the preconditioning).  
For the second case of 280 degrees, the spectrum is also improved by the preconditioning even though it stays quite indefinite.  There is more spacing from the origin and many eigenvalues are in a blob.  GMRES(50) converges very slowly for the easier matrix and not at all for the tougher one, but PP(50)-GMRES(50) converges rapidly for both cases.  See Table~\ref{Tab: Complex} for results with these two matrices and one in-between them.  For the 230 case, there is remarkable improvement with just a degree 5 polynomial.  For the more difficult matrices, a degree 50 polynomial is needed.  Further cases with eigenvalues further extending around the origin require even higher degree polynomials.  For example with a spectrum 290 degrees around the origin, using a degree 50 polynomial takes 13.9 minutes while $d=150$ converges in only 6.4 seconds.  Going to the even more difficult case of 300 degrees, both a high degree polynomial and a larger GMRES restarting parameter are needed as PP(150)-GMRES(50) takes 13.5 minutes but PP(150)-GMRES(150) converges in 10.9 seconds.  Note that polynomials of degree higher than 150 may not be stable in this situation.    


\begin{table}
\caption{Complex matrix of Example 6.  PP(d)-GMRES(50) is used to solve linear equations with eigenvalues through three angles around the origin.   
}
\begin{center}
\begin{tabular}{|c||c|c||c|c||c|c|}  \hline\hline  
 & \multicolumn{2}{c||}{ Angle $230^{\circ}$ }  &   \multicolumn{2}{c|}{Angle $255^{\circ}$}    &   \multicolumn{2}{c|}{Angle $280^{\circ}$}   \\ \hline \hline
d - degree          &  MVP          & time      &  MVP          & time      &  MVP          & time    \\ 
of poly $\varphi$   & (thou's)      &           &  (thou's)     &           &  (thou's)     &            \\   \hline \hline
1 (no pp)           & 2.6e+6        & 7.6 days  &  -            & -         &  -            & -         \\ \hline
5                   & 11.2          & 0.54 sec  &  33,485       & 18 min  &  -            & -         \\ \hline
50                  & 6.4           & 0.28 sec  &  24.4         & 0.45 sec     & 772           & 6.7 sec       \\ \hline

\hline
\end{tabular}
\end{center}
\label{Tab: Complex}
\end{table}

As mentioned earlier, a $\varphi$ polynomial  in the complex plane cannot be zero at the origin and move only toward 1 as it moves away from the origin.  It needs to have both ups and downs while moving away. However, for this example, the spectrum only partly surrounds the origin.  Thus, the polynomial is somewhat able to head towards 1 as it moves away from the origin over the spectrum.  For the case of the spectrum going 230 degrees around, the upper-left plot in Figure~\ref{fig:levcurve} shows contours for the real part of the degree 5 polynomial and the upper-right has the imaginary part.  The eigenvalues are yellow dots.  The polynomial is able to flatten out over most of the spectrum and push much of the eigenvalues to have real parts between 0.5 and 1.0.  Extending further out from this relatively flat portion in the real plot are five valleys and five ridges.  Heading left from the origin is a valley and to the right of the flat area is a rising ridge.  Valleys and ridges alternate going around.  
Next, for the $280^\circ$ matrix, the real contours are in the bottom-left. 
For this more difficult spectrum, the real part is not able to make it to 0.5 except at a few eigenvalues.   A higher degree polynomial can be significantly more effective.  The real part of the degree 50 polynomial is shown in the lower-right part of the figure.  This polynomial goes above 0.5 except for the part of the spectrum near the origin.  There are many valleys and ridges for this polynomial.

Unlike for the previous example, balancing is not effective.  These contour maps help show why.  Balancing keeps the real part of the polynomial from being able to dive down quickly when moving to the left from the origin.  Similarly, moving to the right, the polynomial could not as quickly move up toward 1.  Two upcoming examples (8 and 10) will back up the result that balancing may be detrimental when the spectrum goes much of the way around the origin.    

\begin{figure}[t!]
\begin{center}
\includegraphics[width=4.5in]{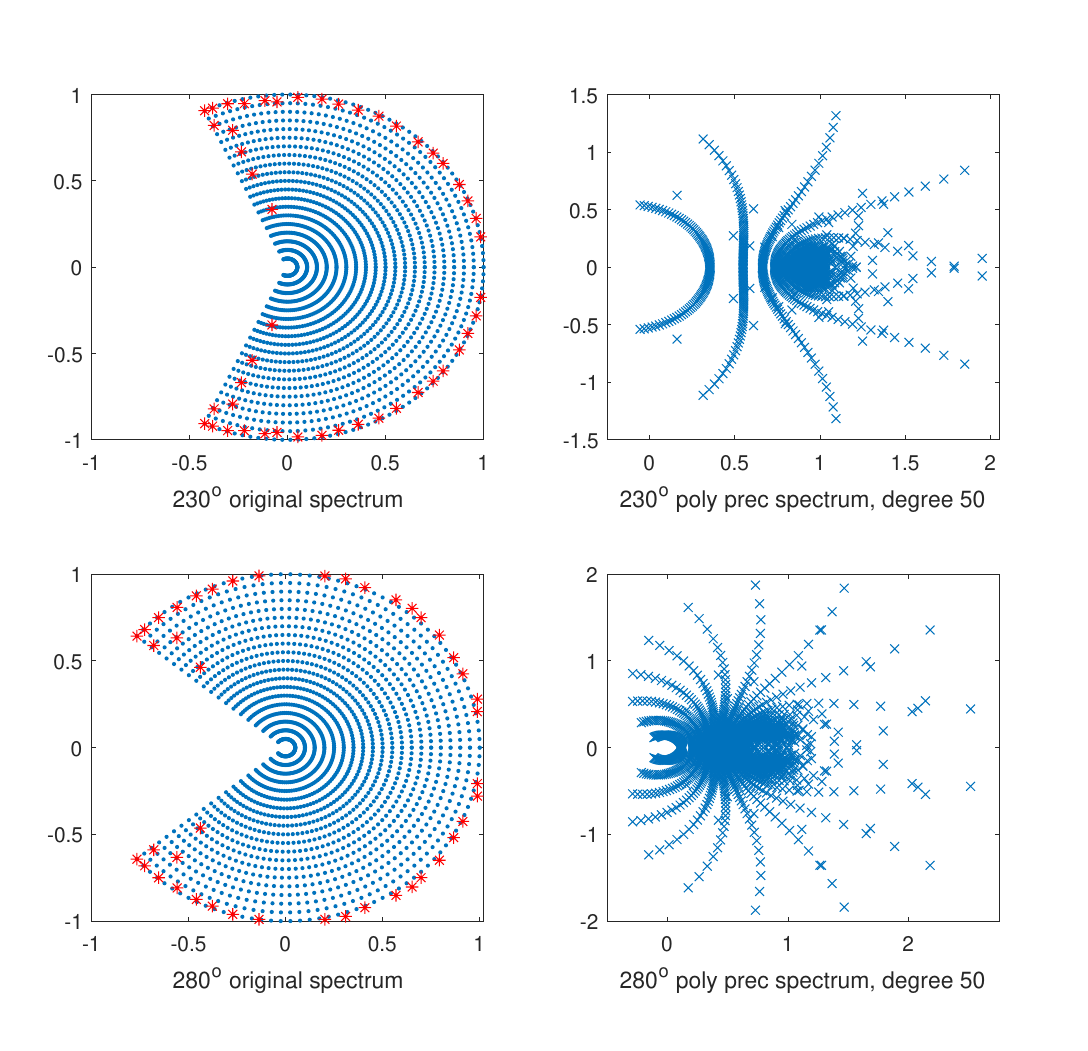}
\end{center}

\vspace*{-20pt}
\caption{\label{fig:complex} The effect of polynomial preconditioning on very complex spectra.  The top two plots are for a matrix with eigenvalues going around $230^{\circ}$, and the bottom plots are for $280^{\circ}$.  On the left side, the original eigenvalues are shown with dots in the complex plane.  The (red) asterisks are the polynomial roots.  On the right two plots, the `x' marks show the eigenvalues of the polynomial preconditioned spectrum of $\varphi(A)$ with degree $d=50$. 
 }
\vspace{-0.3in}
\end{figure}

\begin{figure}[t!]
\begin{center}
\includegraphics[width=4.5in]{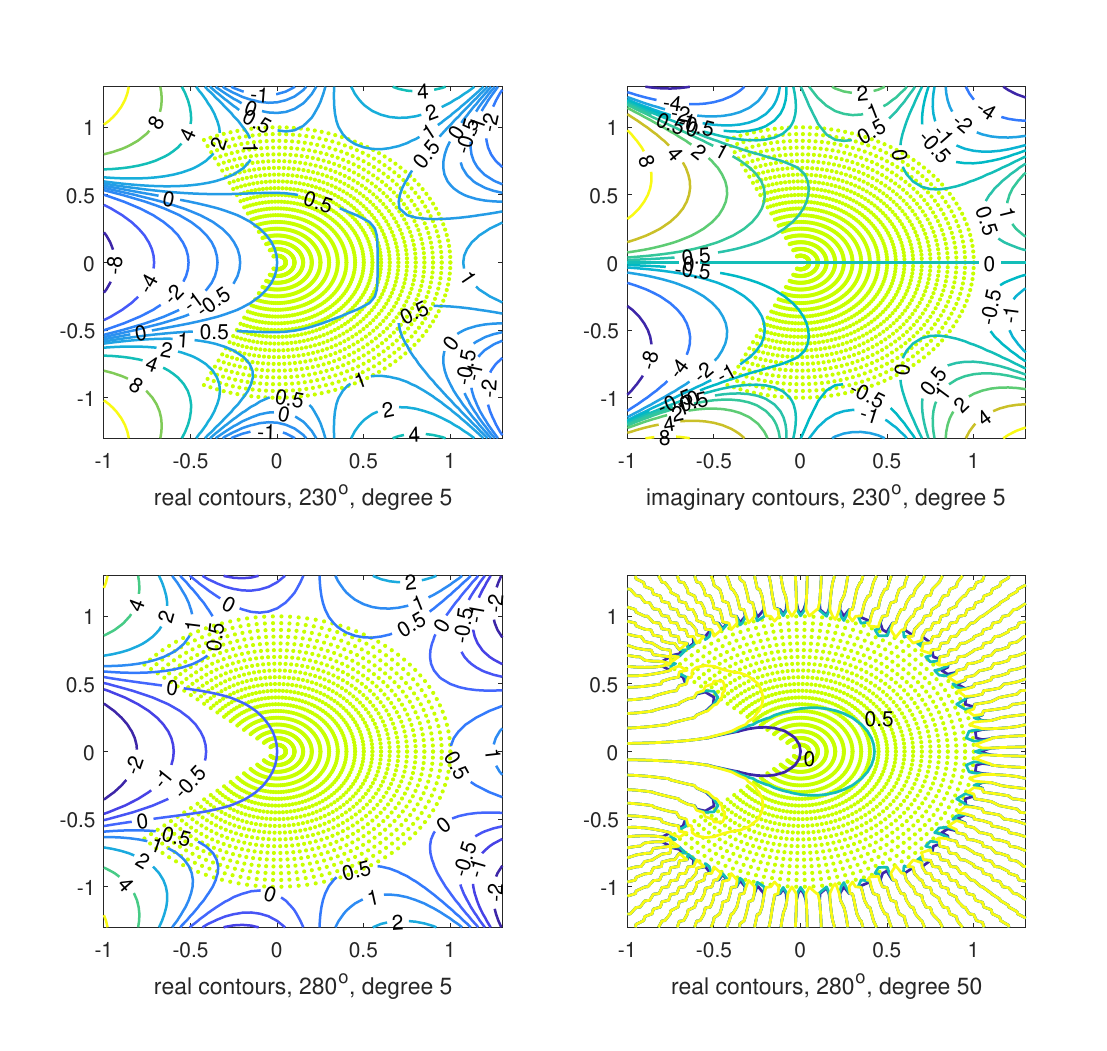}
\end{center}

\vspace*{-20pt}
\caption{\label{fig:levcurve} 
Contour maps for polynomials with the matrices from Example 6. }
\end{figure}

\section{Stability Control}

Here we give stability methods specialized for polynomial preconditioning of indefinite problems.  

Stability of the GMRES residual polynomial is mentioned in Subsection 2.2 (from sources~\cite{PPArn,PPGStable}). 
A $pof$ value is associated with each root $\theta_j$ of the GMRES residual polynomial $\pi$ (harmonic Ritz value), $\pof(\theta_j) = \prod_{i\neq j}\left(1-\frac{\theta_j}{\theta_i}\right)$.  A high $pof$ value indicates a root that stands out from the others and indicates that there may be instability (unless the root is spurious).
To fix this, extra copies of these outstanding roots are added to the $\pi$ polynomial.  However, for a small root $\theta_i$, the linear factor $(I - \frac{A}{\theta_i})$ from an extra copy can blow up the polynomial at large portions of the spectrum.  For an indefinite spectrum, there can be outstanding eigenvalues on one side of the spectrum that are much smaller than eigenvalues on the other side.  For this situation where there are high $pof$'s on the shorter side of the spectrum, we suggest alternatives to adding extra roots.  

The following theorem shows a correlation between a high $pof$ value and the accuracy of the corresponding harmonic Ritz vector.  It is assumed that the root of interest is equal to an eigenvalue, since a large $pof$ will usually correspond to a root that closely approximates an outstanding eigenvalue (for spurious harmonic Ritz values, we do not need to deal with them, as is mentioned right before Algorithm 6).

\begin{theorem}
      Assume a GMRES residual polynomial $\pi$ is generated with $d$ iterations of solving $Ax=b$, where $A$ is diagonalizable and $b$ is norm one.  Let the roots of $\pi$ (the harmonic Ritz values) be $\theta_i$, for $i=1,\ldots,d.$  Let the eigenvectors of $A$ be $z_i$'s and let $Z$ be the matrix with these as columns.
      Let $b = \sum_{i=1}^n \beta_i z_i.$ 
    Assume a particular Ritz value $\theta_j$ equals an eigenvalue $\lambda_j$.  Then $$||r_j||\leq\frac{|\theta_j| \|Z^{-1}\|}{|\beta_j|pof(\theta_j)},$$
where 
$pof(\theta_j) =  \prod_{i\neq j} \left(1-\frac{\theta_j}{\theta_i}\right).$
\end{theorem}

\begin{proof}
Let $y_j$ be a harmonic Ritz vector. 
From~\cite{IEN,GMRES-IR}, it can be written as \[y_j=\omega_j(A)b=s_j\prod_{i\neq j}\left(I-\frac{A}{\theta_i}\right)b \label{eq:yNorm} \tag{5.1},\]
where $s_j$ is the constant that normalizes $\omega_j$ at $\theta_j$, meaning $s_j=\displaystyle\frac{1}{\prod_{i\neq j}\left(1-\frac{\theta_j}{\theta_i}\right)}$.

Multiply \eqref{eq:yNorm} by $(I-\frac{A}{\theta_j})$, use that $\pi(A) = \prod_{i=1}^{n}(1-\frac{A}{\theta_j})$, and rearrange: $$ Ay_i-\theta_jy_j = -\theta_js_j\pi(A)b.$$
Since the GMRES residual norm cannot rise, $\|b\| = 1$ implies that $\|\pi(A)b\| \le 1$.  
So \[||Ay_j-\theta_jy_j||=|\theta_j||s_j|\|\pi(A)b\| \leq |\theta_j||s_j|=\frac{|\theta_j|}{pof(\theta_j)}. \label{eq:normIneq} \tag{5.2}\]
Next, 
\begin{multline*}
||y_j|| = \|Z [\beta_1 \omega_j(\lambda_1), \ldots, \beta_n \omega_j(\lambda_n)]^T\| \geq \frac{1}{\|Z^{-1}\|} \|[\beta_1 \omega_j(\lambda_1), \ldots, \beta_n \omega_j(\lambda_n)]^T\| \\ \geq \frac{1}{\|Z^{-1}\|} |\beta_j| |\omega_j(\lambda_j)| \geq \frac{1}{\|Z^{-1}\|} |\beta_j|, 
\end{multline*}{}
using that $\theta_j = \lambda_j$ and $\omega_j$ is normalized to be one at $\theta_j$. 
Combined with \eqref{eq:normIneq}, we have
$$||r_j||=\frac{||Ay_j-\theta_jy_j||}{||y_j||}\leq \frac{|\theta_j| \|Z^{-1}\|}{|\beta_j|pof(\theta_j)}.$$
    
\end{proof}

This theorem indicates that for an accurate $\theta_j$ with a large \textit{pof}, $y_j$ will generally be a good approximate eigenvector.  This motivates one of the stability procedures that follow, where approximate eigenvectors are used for deflation. 

Algorithm~\ref{alg:StContr} gives our approach for dealing with an unstable polynomial.  It has the previous way of adding extra copies of roots to outstanding eigenvalues, but now only to the side in the complex plane where the spectrum is larger.  Then it has corrections for the other side if there are outstanding eigenvalues there.  The Step 5.a.\ correction uses that instability error on the smaller side is mostly in the direction of just a few eigencomponents corresponding to roots with high $\pof$ values.  So projecting over a few eigenvectors reduces the rogue eigencomponents.  Likewise in Step 5.b., it generally only takes a few iterations of GMRES to improve the residual vector.  

In the algorithm, it is important to identify if there are large spurious harmonic Ritz values.  These are much more likely to occur for indefinite problems than for definite (large spurious harmonic Ritz values correspond to spurious Ritz values of $A^{-1}$ near the origin~\cite{IE}, which for symmetric problems can happen only in the indefinite case).  Extra roots for stability control are not needed at spurious harmonic Ritz values.  If there is a high $\pof$ for a particular large value, then the residual norm can tell us whether or not it is spurious.

\begin{algorithm}
\caption{Choice of Stability Control for an Indefinite Spectrum}\label{alg:StContr}
\begin{description}
 \item[0.] Compute $\pof$ values for the polynomial roots.  Set $\pofcutoff$ (mentioned in Subsection 2.2) and $\textit{rncutoff}$ (used to indicate a reasonably accurate approximate eigenvector).  Determine which is the larger side of the spectrum in the complex plane; this is the side for which the harmonic Ritz values (with spurious values not counted) extend the furthest from the imaginary axis.  Alternatively, Ritz values can be used and then large spurious values are unlikely.
 \item[1.] Optional: Add a limited number of extra copies of roots on the smaller side of the spectrum that have $\pof > \pofcutoff$ and $\|r_j\| \le \textit{rncutoff}$.  Then recompute $\pof$ values. 
 \item[2.] If the largest $\pof$ on the small side of spectrum is greater than $10^{20}$, reduce the degree of the polynomial and begin again.
 \item[3.] Add extra copies of roots on the larger side if $\pof > \pofcutoff$.  The number of copies at a root is the least integer greater than $(log10(\pof(k)) - \pofcutoff) / 14$.
 \item[4.] Apply PP-GMRES.  Solve until a shortcut residual norm reaches the desired tolerance.  If the actual residual norm is not as accurate, then the correction methods can be applied.
 \item[5.] Correction phase: Apply one or both correction steps.
 \begin{itemize}
 \item[a.] Apply Galerkin projection (Algorithm~2 in \cite{MGLE}) over the span of the approximate eigenvectors corresponding to the roots on the smaller side of the spectrum that have both $\pof(\theta_j) \ge \pofcutoff$ and $\|r_j\| \le \textit{rncutoff}$.  This deflates these eigencomponents from the residual vector.
\item[b.] Run GMRES (without polynomial preconditioning) for a few iterations.
 \end{itemize}
\end{description}
\end{algorithm}

In the optional Step 1.\ of the algorithm, perhaps only one root should be added to the shorter side.  Or only add roots that are within half the magnitude of the largest non-spurious harmonic Ritz value on the other side. The effect of adding these roots on the polynomial on the large side can be checked with the Hermite spline polynomials in Subsection 3.3.


{\it Example 7.}  The stability algorithm is tested with a matrix that has outstanding eigenvalues on both sides of the complex plane.  However, the outstanding eigenvalues on the left side are much smaller in magnitude than those on the other side.  
The matrix is bidiagonal of size $n=5000$ with diagonal elements $-500, -400, -300, -200,$ $ -100, 0.001,0.01, 0.02, 0.03, \ldots, 0.09,0.1, 0.02, 0.03,$ $ \ldots, 0.9,1, 2, 3, \ldots 4971, 5000, 5100,$ $ 5200,$ $ 5300, 5400$ and with superdiagonal elements all 0.1.  The right-hand side is generated random normal, and then is normed to one.  PP(d)-GMRES is applied using the stability control in Algorithm~\ref{alg:StContr} with roots added only on the larger (right) side (i.e. Step 1 is not activated).  No balancing is used.  Linear equations are solved until a shortcut residual norm reaches $10^{-10}$.  The $\textit{rncutoff}$ value is not needed because no spurious harmonic Ritz values occur.  The value $\pofcutoff = 10^6$ is used.   The correction in Step 5.a.\ has different numbers of approximate eigenvectors.  For example, with degree 57, only two have reached $\pof > 10^{6}$ while for degree 100 there are five.  For the GMRES correction in Step 5.b., 10 GMRES iterations are used.   
Table~\ref{Tab:Corrections} has results with three types of correction.  First, the eigenvector deflation is in the fifth column, then the next columns have the GMRES correction, followed by both (deflation then GMRES).  Much more accurate results are produced.  Using both corrections is usually better, but may not be needed.    

Because of the difficulty of this problem, polynomial preconditioning is very important.  In fact, PP-GMRES does not converge until the degree of the polynomial reaches 57 (and not for every degree just larger than that).  

We finish this example by activating the optional Step 1 for the polynomial of degree initially 115.  For this high degree polynomial, the method does not converge even in the shortcut residual norm.  We need to activate Step 1 of the algorithm in order to keep the polynomial more controlled on the short side.  Here we add one root copy to each root on the left side with $pof$ over $10^{14}$.  This results in three extra roots.  Then the rest of the algorithm is applied along with the GMRES correction.  The final residual norm is $1.1 * 10^{-10}$.  So by using Step 1, we are able to get an accurate result.  However, we do not succeed with polynomials of much higher degree.

\begin{table}
\caption{Bidiagonal matrix of Example {7}.  PP(d)-GMRES(50) is run followed by correction procedures. 
}
\begin{center}
\begin{tabular}{|c||c|c|c|c|c|c|}  \hline\hline

d - degree          & Time          & max       &  Res Norm     & Deflate       & GMRES     & Both      \\ 
of poly $\varphi$   & (sec's)       & pof       &  w/o correct  & correction    & correct   & corrects  \\   \hline \hline
57 + 2              & 44            & 2.9e+9    & 5.8e-6        & 1.1e-10       & 9.6e-11   & 9.6e-11   \\ \hline
75 + 4              & 0.92          & 1.7e+14   & 1.9e+2        & 2.8e-10       & 1.1e-10   & 5.1e-11   \\ \hline
90 + 4              & 0.97          & 2.5e+18   & 8.3e+4        & 1.5e-9        & 2.5e-9    & 4.2e-11   \\ \hline
100 + 5             & 0.92          & 1.1e+21   & 1.7e+5        & 2.3e-9        & 3.2e-10    & 3.8e-11   \\ \hline 
110 + 5             & 3.5           & 4.6e+23   & 1.0e+9        & 1.4e-5        & 6.3e-8    & 1.2e-7    \\ \hline 
115 + 5             & -             & 9.0e+24   & - & - & - & -       \\ \hline
\hline
\end{tabular}
\end{center}
\label{Tab:Corrections}
\end{table}

{\it Example 8.} The matrix is cz20468 from the SuiteSparse Matrix Collection.  Standard GMRES(50) fails to converge, prompting the use of preconditioning techniques.  ILU factorization (with \textit{droptol}=0.01) is applied as an initial preconditioner.  This produces a preconditioned operator with an indefinite spectrum, for which GMRES(50) still does not convergence.  We add polynomial preconditioning and now can get convergence; see Table~\ref{Tab:Cz}.

For this example, a solution with residual of $10^{-8}$ is sought, so PPGMRES is run until the shortcut residual norm produces a solution of $10^{-10}$ to compensate for potential instability in the polynomial preconditioner. The parameters are set to \textit{pofcutoff} $ = 10^4$ and \textit{rncutoff} $ = 10^{-3}$. Results for polynomial degrees ranging from 5 to 40 are presented in Table~\ref{Tab:Cz}. The higher the degree of the polynomial, the more unstable the polynomial can be. The max $pof$ column provides a glimpse of this as some $pof$'s can become greater than $10^{20}$ for degrees 35 and higher. These high $pof$'s are observed on the larger side, while the pofs on the shorter side remain below $10^{20}$ as stipulated in Algorithm~\ref{alg:StContr}.

At higher polynomial degrees, more vectors are available for deflation. For the degree 20 polynomial only one approximate eigenvector is deflated, while 7 are used for the degree 40 polynomial.  Applications of deflation and/or GMRES correction consistently restore accuracy to the solution. Notably, 10 iterations of GMRES alone often serves as a sufficient correction and even outperform using both corrections for degrees 20 and 30. The bottom two rows of Table~\ref{Tab:Cz} show the degree 40 polynomial with and without the optional step 1 in Algorithm~\ref{alg:StContr}. Applying Step 1 increases the initial accuracy of the residual from $7.6*10^9$ to $810$ which allows the accuracy to reach the desired goal of $10^{-8}$ post deflation.

\begin{table}
\caption{cz20468 matrix with ILU factorization (droptol=0.01) of Example {8}.  PP(d)-GMRES(50) is run followed by correction procedures.}
\begin{center}
\begin{tabular}{|c||c|c|c|c|c|c|}  \hline\hline

d - degree          & Time          & max       &  Res Norm      & Deflate      & GMRES         & Both            \\ 
of poly $\varphi$   & (sec's)       & $pof$       &  w/o correct   & correction   & correction    & corrections                 \\   \hline \hline
1                   & -             &           &                &              &               &                              \\ \hline
5 + 0               & 173           & 5.7       & 5.6e-9         & -            & 2.7e-10       & -                             \\ \hline
10 + 0              & 166           & 5.8e+2    & 2.3e-8         & -            & 2.1e-10       & -                              \\ \hline
20 + 2              & 2186          & 7.5e+7    & 3.7e-7         & 1.6e-8       & 1.8e-10       & 3.3e-10                         \\ \hline
30 + 5              & 404           & 3.7e+19   & 1.6            & 1.1e-8       & 2.9e-10       & 3.9e-10                          \\ \hline
35 + 6              & 333           & 1.4e+24   & 3.2e+4         & 9.4e-9       & 1.9e-9        & 1.2e-9                            \\ \hline 
40 + 7              & 166           & 2.0e+29   & 7.6e+9        & 3.7e-4       & 5.6e-4        & 6.6e-6                             \\ \hline 
\hline 
40 + 8              & 162           & 2.0e+29   & 8.1e+2         & 1.9e-7       & 7.5e-8        & 4.3e-8                               \\ \hline
\hline
\end{tabular}
\end{center}
\label{Tab:Cz}
\end{table}

\section{Convergence Estimates}

This section develops estimates for the effectiveness of polynomial preconditioned GMRES for indefinite problems. Using Chebyshev polynomial results, we derive theoretical estimates on how well polynomial preconditioning can enhance the convergence of restarted GMRES. The analysis reveals a significant reduction in the required number of matrix-vector products under idealized conditions.
We assume that all polynomials can be approximated by Chebyshev polynomials, including both the polynomial for preconditioning and the polynomials that underlie the GMRES method.  We assume the spectrum of the matrix satisfies $\sigma(A)\subset [u,v]\cup[a,b]\subset\mathds{R}$ where $u\ll v<0<a\ll b$, and with the longer interval on the right of the origin.  It is also assumed that
\begin{equation}  \label{approx:1}
    T_m(1+\delta)\doteq 1+m^2\delta
\end{equation}
where $T_m$ is the standard standard Chebyshev polynomial of the first kind of degree $m$ and $0<\delta\ll1$. 

To approximate the GMRES polynomial that is being used for polynomial preconditioning, we will use a composite polynomial.  For a spectrum that is about the same on both sides of the origin, one could use an inner polynomial that is a quadratic and maps to a positive definite spectrum.  Then at that point a standard shifted and scaled Chebyshev polynomial can be applied as the outer polynomial (see \cite[Thm.~6]{IE} for such an approach).  This makes it possible to estimate the success of polynomial preconditioning.  We skip this development with a quadratic and jump ahead to having a cubic as the inner polynomial.  This is better for lopsided spectra that extend further on one side of the origin than the other. The cubic is best for spectra that extend about three times further on one side than the other.  Higher degree inner polynomials could be used for more lopsided spectra.

We first develop a degree 3 polynomial $f(x)$ which maps $[u,v]\cup[a,b]$ onto the interval $[-1,1]$ while ensuring $f(0) > 1$. Consider  $h(x)=(x-a)(x-b)(x-v)$, which has roots at $x=a,b,v$ and local maximum at $\gamma_1=\frac{(a+b+v)-\sqrt{a^2+b^2+v^2-(ab+av+bv)}}{3}$, and a local minimum at $\gamma_2=\frac{(a+b+v)+\sqrt{a^2+b^2+v^2-(ab+av+bv)}}{3}$. It can be shown that $\gamma_1\in[v,a]$ and $\gamma_2\in[a,b]$. It can also be shown that $h(\gamma_2)=h(a+b+v-2\gamma_2)<0$, so for $u<a+b+v-2\gamma_2$, then $h(u)<h(\gamma_2)$. With this, we can define:
\begin{center}
    \[f(x)= \begin{cases} 
      -\frac{2h(x)}{h(u)}+1 & u \leq a+b+v-2\gamma_2 \\
      -\frac{2h(x)}{h(\gamma_2)}+1 & u > a+b+v-2\gamma_2 \\
   \end{cases}
\]
\end{center}

In the case where $u > a+b+v-2\gamma_2$, $$f(0)=\frac{2abv}{(\gamma_2-a)(\gamma_2-b)(\gamma_2-v)}+1\doteq \frac{2abv}{(\frac{-2b}{3}-a)(\frac{-2b}{3}-b)(\frac{-2b}{3}-v)}+1\doteq\frac{-27av}{2b^2}+1$$
and if $u\leq a+b+v-2\gamma_2$ $$f(0)=\frac{2abv}{(u-a)(u-b)(u-v)}+1\doteq \frac{2abv}{u^2(u-b)}+1.$$
In both cases, we conclude that $f(0)=1+\delta$ with $\delta$ small, allowing us to utilize $(\ref{approx:1})$ above.

Now, composing $f(x)$ with the Chebychev polynomial $T_{\frac{m}{3}}$, we obtain: $\frac{T_{\frac{m}{3}}(f(x))}{T_{\frac{m}{3}}(f(0))}$, which forms the desired polynomial of degree $m$. The maximum value of this Chebyshev polynomial over $[u,v]\cup[a,b]$ is: $\frac{1}{T_{\frac{m}{3}}\left(1+\frac{2abv}{(\xi-a)(\xi-b)(\xi-v)}\right)}$, where $\xi=\gamma_2$ or $\xi=u$. This quantity estimates the improvement in residual norm per one cycle of GMRES$(m)$. Using approximation $(\ref{approx:1})$, for either value of $\xi$, the residual norm is improved for approximately:
\begin{equation} \label{approx:2}
    \frac{1}{1+\delta(\frac{m}{3})^2}\doteq 1-\frac{m^2}{9}\delta. \nonumber
\end{equation}
For $d$ cycles of GMRES$(m)$, the improvement factor is approximately: 
\begin{equation}  \label{approx:3}
    \left(1-\frac{m^2}{9}\delta\right)^d\doteq 1-\frac{d m^2}{9}\delta.
\end{equation}

We view a single cycle of PP($d$)-GMRES($m$) as a composition of two polynomials: the preconditioner polynomial from GMRES$(d)$ on the inside and the GMRES$(m)$ polynomial on the outside. This can be modeled by comparing shifted and scaled Chebyshev polynomials, leading to the residual improvement estimate:
\begin{equation}  \label{approx:4}
    \frac{1}{T_m\left(T_{\frac{d}{3}}\left(1+\frac{2abv}{(\xi-a)(\xi-b)(\xi-v)}\right)\right)}\doteq \frac{1}{T_m\left(1+\delta\left(\frac{d}{3}^2\right)\right)}\doteq \frac{1}{1+\delta\left(\frac{m^2d^2}{9}\right)}\doteq 1-\frac{d^2m^2}{9}\delta
\end{equation}

Comparing the improvements from $(\ref{approx:3})$ and $(\ref{approx:4})$, we conclude that polynomial preconditioned GMRES$(d)$ converges approximately $d$ times faster in terms of matrix-vector products.\\

\section{Interior Eigenvalue Problems}

Computing eigenvalues and eigenvectors of large matrices is another important task in linear algebra.  The standard approach is restarted Arnoldi~\cite{So,Arnoldi-R,WuSi,HRAM} or restarted Lanczos for symmetric~\cite{WuSi,Lan-DR} (nonrestarted CG~\cite{HeSt} can be used in LOBPCG~\cite{KnNe} for symmetric problems).
Polynomial preconditioning is even more important for large eigenvalue problems than it is for linear equations.  Eigenvalue problems tend to be difficult, because standard preconditioning is less effective for eigenvalue problems and often not used.  Standard preconditioning can be incorporated into more complicated methods such as LOBPCG, Jacobi-Davidson~\cite{SlvdV} and Preconditioned Lanczos~\cite{PL}.  However, then only one eigenvalue can be targeted at a time~\cite{pescs}.

Eigenvalue polynomial preconditioning using a GMRES polynomial is given in~\cite{PPArn}.  Here we focus on the case where the desired eigenvalues are in the interior of the spectrum.   
The polynomial for preconditioning interior eigenvalue problems is found similarly as for indefinite system of linear equations.  GMRES is applied to a shifted problem in order to generate a polynomial that targets a certain zone.  

Here we give another balancing method that balances on an interval, meaning that the value of the polynomial $\varphi$ is equal at the two ends of the interval (see~\cite{LiXiVeYaSa} for this type of balancing of a Chebyshev polynomial for symmetric eigenvalues).  Unlike the other balancings, it does not attempt to make the derivative of the $\phi$ polynomial zero at a point. However, like Balance Method 1, it is done by adding a root.   We give this Balance Method 5 for the case of balancing around the origin, so on the interval [-a a], but this can be shifted for an interval around another spot.  

\begin{algorithm}
\caption{Balance Method 5: Add a root for balancing on the interval $[-a \ a]$.}\label{alg:bal5}
\begin{description}
 \item[0.] Let polynomial $\varphi(\alpha)$ of degree $d$ correspond to $\pi(\alpha)$  with roots $\theta_1, \ldots, \theta_d$.
 \item[1.] Compute $\beta = \frac{\varphi(a)}{\varphi(-a)},$ then $\eta = a \frac{(1 + \beta)}{(1 + \beta)}$.  Add $\eta$ to the list of polynomial roots.
 \item[2.] The new polynomial of degree $d+1$ is $\varphi_5(\alpha) = 1-\pi(\alpha)*\Big(1 - \frac{\alpha}{\eta}\Big)$.
\end{description}
\end{algorithm}

Balancing the polynomial is important for interior eigenvalue problems because otherwise a lopsided number of eigenvalues may be found on one side of the requested target.  Here we focus on Balance Method 1, but others could be used.  In particular, Balance Method 5 gives essentially the same results as Method 1, so only Method 1 is given in the results.

{\it Example 9.}  
We use a diagonal matrix with eigenvalues $1, 2, 3, \ldots, 499, 500, 500.2,$ $500.4, \ldots, 519.8, 520, 521, 522, \ldots, 4920$.  So $n = 5000$, and there is a group of 101 eigenvalues close together starting at 500.  We seek the 30 eigenvalues nearest 500.33.  Arnoldi(80,40) is used, so the Krylov subspace is built out to dimension 80 and then restarted with 40 vectors.  Reorthogonalization is applied in this and the other examples of eigenvalue computation in this section.  The residual norm tolerance is set to $10^{-8}$.  Table~\ref{Tab:IntEV} has results with polynomial preconditioning of degrees 10, 50 and 100, both with and without balancing.  This is compared to no polynomial preconditioning (note that for some of our interior eigenvalue testing, harmonic Rayleigh-Ritz~\cite{IE,PaPavdV,IEN,HRAM} performs better, but for this particular test regular Rayleigh-Ritz is best and is reported).  The degree 10 polynomial dramatically speeds up the eigenvalue computation and does not need balancing.  Degree 50 is even better (about two orders of magnitude faster than without polynomial preconditioning), and with balancing it finds the eigenvalues closest to the requested value of 500.33.  Without balancing, the polynomial dips below the axis to the left of the origin and so eigenvalues are found much further to the left of 500.33 than to the right.  Also some are missed where the polynomial takes them well below zero (our program finds the ones closest to zero).  With balancing, the correct eigenvalues are quickly found.  The degree 100 polynomial also has problems without balancing, but with balancing it is too volatile at the largest eigenvalues (among the 30 eigenvalues it finds are the large ones 4843, 4846 and 4870).  So here it is best to use a lower degree polynomial.  If a high degree polynomial is desired then Balance Method 4 could be applied.

\begin{table}
\caption{Diagonal Matrix with $n=5000$.  PP(d)-Arnoldi(80,40) is used to find 30 interior eigenvalues near 500.33.  Degrees of the balanced polynomial are one higher than indicated because of the added root.  Residual tolerance is $10^{-8}$.  The asterisks indicate max residual norm only reaches $5.3 * 10^{-8}$. }
\begin{center}
\begin{tabular}{|c||c|c|c||c|c|c|}  \hline\hline
& \multicolumn{3}{|c||}{Unbalanced GMRES Polynomial}       &  \multicolumn{3}{|c|}{Balanced Polynomial }  \\   \hline
d, deg                 & Time       & MVP's     & Eigenvalues   & Time      & MVP's     & Eigenvalues   \\ 
of $\varphi$           & (sec's)    & (tho's)   &               &  (sec's)  & (tho's)   &    \\   \hline \hline
1 (no pp) & 363         & 118       & 496 - 505     &           &  &      \\  \hline
10        & 19          & 61        & 496 - 505     & 20    & 70  & 496 - 505    \\ \hline 
50                      & 3.1       & 38.1      & 478 - 503.4    & 4.4      & 57.2      & 496 - 505      \\ 
   &           &           & (missing some)  &    &   &   \\   \hline 
100       & 1.7     & 32     &  472 - 503.6  & $7.2^*$  & $166^*$  & 496 - 504.4    \\ 
   &           &           & (missing some)  &         &   & + 3 large ones \\   \hline  
\hline
\end{tabular}
\end{center}
\label{Tab:IntEV}
\end{table}

We next consider two tests of computing interior eigenvalues with matrices from applications.

{\it Example 10.} 
The matrix Af23560 is part of the package NEP~\cite{BDDD} (Non-Hermitian Eigenvalue Problem) and is available from SuiteSparse.  The matrix is of size $n=23{,}560$ and was developed for stability analysis of Navier-Stokes equations.  The upper left portion of Figure~\ref{fig:AfB} has the overall spectrum, which is very complex and lies in the left half of the complex plane.  

We target the eigenvalues nearest $-4$ using Arnoldi(600,300).  This is an interior eigenvalue problem, because the 300 nearest eigenvalues do not include the right-most (exterior) ones.  
We compare a polynomial preconditioner of degree 50 to no preconditioner.  The second picture down on the left of Figure~\ref{fig:AfB} has the preconditioned spectrum with no balancing.  Note it is very indefinite.

PP(50)-Arnoldi without balancing is run for 25 cycles and finds 284 Ritz pairs to residual norm below $10^{-3}$.  This takes 32 minutes.  The 284 Ritz values are plotted with (blue) x's on~\ref{fig:AfA}.
Regular restarted Arnoldi finds a similar number of Ritz pairs below $10^{-3}$ in residual norm, specifically 286, with 125 cycles.  This takes 2.6 hours, five times longer, while using less matrix-vector products.  However, the important point is that regular Arnoldi solves an easier problem, focusing more on the eigenvalues nearer the edge of the spectrum.  Even if it is run longer, 467 cycles (9.7 hours) so that the first $200$ residual norms are below $10^{-7}$, the eigenvalues found are still mostly to one side of the target.  Figure~\ref{fig:AfA} shows with (red) circles the 298 Ritz values from this run that have residual norms below $10^{-3}$.  This shows that polynomial preconditioned Arnoldi does better at finding the eigenvalues nearest the target. 

If Balance Method 1 is applied to the degree 50 polynomial, the results are worse: only 252 eigenvalues are accurate to $10^{-3}$ after 25 cycles.  Figure~\ref{fig:AfB} shows how the spectrum is changed with balancing.  The upper right portion has the 50 eigenvalues of $A$ numbered in order of distance from $-4$.  The lower left has where these move to with the preconditioned spectrum.  Then the lower right has them with balancing added.  The balancing does succeed in moving the real eigenvalues on the left side over to the right, so that the real portion of the spectrum looks positive definite.  However, the rest of the spectrum does not move in a predictable fashion.  Also, note these 50 eigenvalues are much closer together after balancing and thus harder to find.  The non-balanced polynomial goes through the desired region with a slope and so does not push the eigenvalues together.

\begin{figure}[t!]
\begin{center}
\includegraphics[width=4.5in]{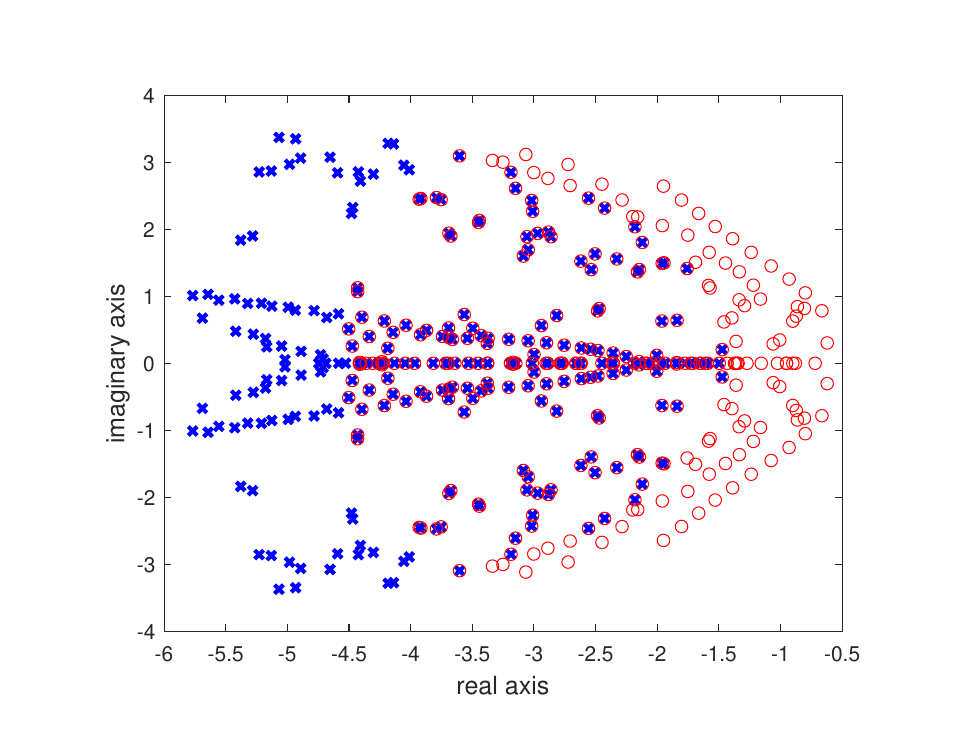}
\end{center}

\vspace*{-20pt}
\caption{\label{fig:AfA} 
Matrix Af23560 with a target of finding eigenvalues near -4.  Ritz values computed by regular Arnoldi are (red) circles, while (blue) x's show values from PP(50)-Arnoldi.  Polynomial preconditioning is much better at finding eigenvalues around the target. }
\vspace{-0.2in}
\end{figure}

\begin{figure}[t!]
\begin{center}
\includegraphics[width=4.5in]{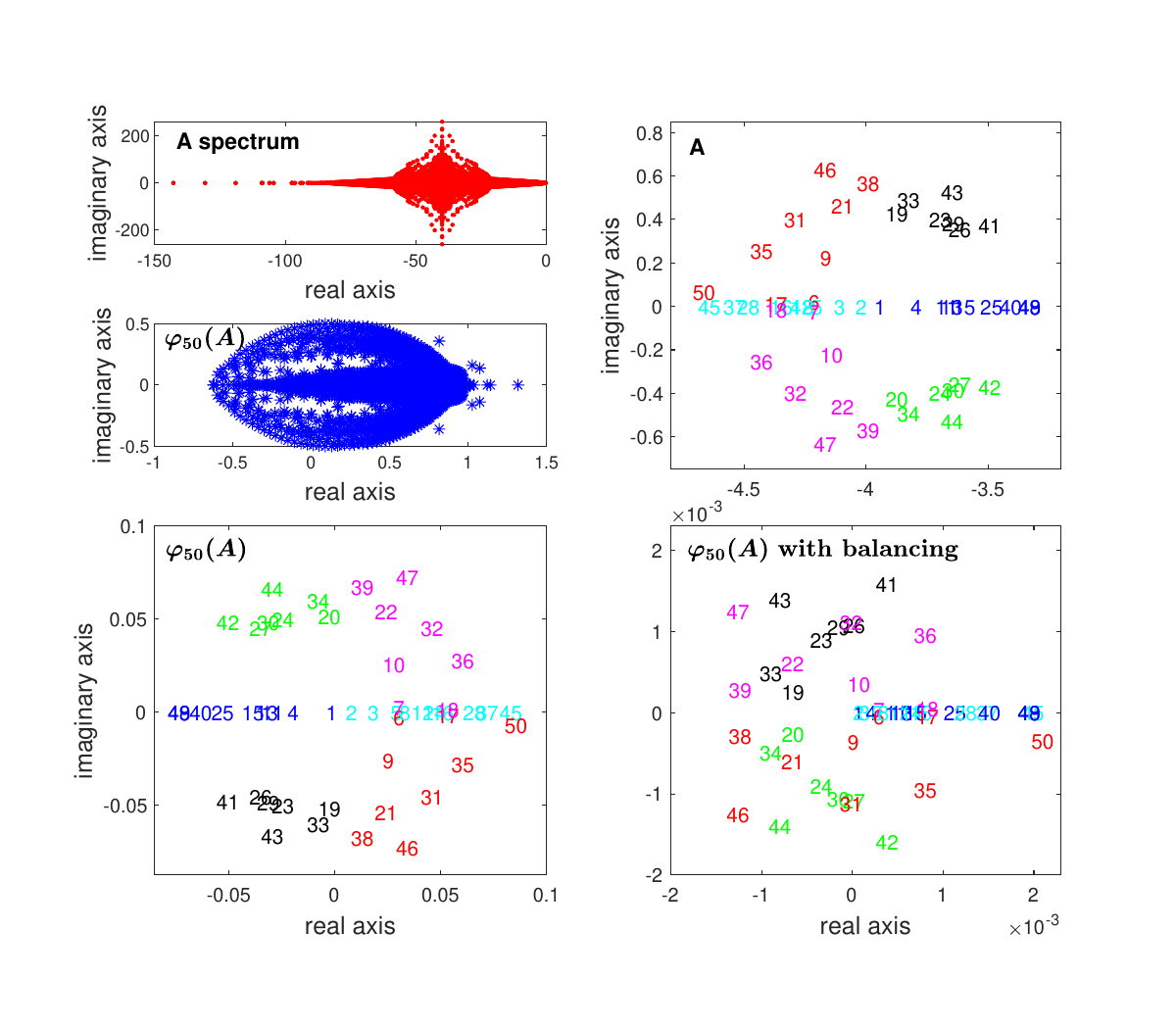}
\end{center}

\vspace*{-20pt}
\caption{\label{fig:AfB} 
The upper left has the spectrum of Af23560 in the complex plane, and just below is the polynomial preconditioned spectrum with no balancing. 
The upper right of the figure has eigenvalues of $A$ numbered in order of distance from $-4$.  The lower left has where these move for the preconditioned spectrum with polynomial of degree 50.  Then the lower right has balancing added.  }
\end{figure}

{\it Example 11.} Large, complex, non-Hermitian matrices arise in quantum chromodynamics (QCD).  However, they can be transformed into Hermitian by multiplying by the QCD $\gamma_5$ matrix.  The spectrum is then real and indefinite with about the same spread of eigenvalues on both sides of the origin.  We compute 15 eigenvalues and eigenvectors around the origin for a matrix of size $n=3.98$ million.  These can be used to deflate solution of linear equations with the conjugate gradient method~\cite{Lan-DR}.  Table~\ref{Tab:Qcdg5} compares restarted Arnoldi(60,25) with and without polynomial preconditioning.  The polynomial is balanced with Balance Method 1.  For this test, it does not make much difference, but we have seen situations where balancing is needed for QCD problems.  The polynomial preconditioning reduces the time dramatically and makes using eigenvalue deflation much more practical for QCD calculations.

\begin{table}[h]
\caption{A quenched QCD matrix of dimension $n = 3.98$ million is from a $24^4$ lattice at zero quark mass.  The 15 eigenvalues nearest the origin are computed with Arnoldi(60,25) and stopped when 15 eigenvalues reach residual norms below $10^{-8}$. }
\begin{center}
\begin{tabular}{|c|c|}  \hline 
degree of polynomial    & time in hours     \\ \hline \hline
no polynomial           &   149             \\ \hline
 d = 25+1               &   7.82            \\ \hline
 d = 100+1              &   5.49            \\ \hline
\end{tabular}
\end{center}
\label{Tab:Qcdg5}
\end{table}

\section{Conclusion}

Polynomial preconditioning is especially important for difficult indefinite linear equations and interior eigenvalues.  Standard preconditioning is generally less effective for such indefinite problems so polynomial preconditioning can assist.

The work in this paper makes polynomial preconditioning more practical for general matrices.  Polynomial preconditioning faces special challenges for indefinite problems.  These are addressed here, first with methods to balance the polynomial with the goal of creating a definite spectrum.  
Then an approach is given for making the polynomial stable for indefinite problems.  In this, the previous stability method of adding extra copies of roots is applied only to one side of the spectrum.  For the other side, corrections are given using eigenvalue deflation and using GMRES iterations.

Looking forward, the parallelizabilty of polynomial preconditioners makes them well-suited for modern high-performance computing environments, including multi-core processors and GPUs. In some cases, polynomial preconditioners may even replace traditional methods, because polynomials are easier to parallelize than incomplete factorizations. The effects of PPGMRES on GPU architectures should be further investigated to fully understand its potential for accelerating large computations. Future work should also focus on further researching the effectiveness of polynomial preconditioning for interior eigenvalue problems and comparing performance across a wider range of problems.

\section*{Acknowledgments}  The authors would like to thank Mark Embree for providing the Hatano-Nelson matrix in Example 5.

\bibliographystyle{siam}

\end{document}